\documentclass[11pt]{article}

\usepackage{amssymb}
\usepackage{amsthm}
\usepackage{amsmath}
\usepackage{amssymb}
\usepackage{amsfonts}
\usepackage{multirow}
\usepackage{amsthm}
\usepackage{authblk}
\usepackage{xcolor}
\usepackage{graphicx}
\usepackage{slashbox}
\usepackage{natbib}
\usepackage[a4paper,textwidth=15cm]{geometry}
\usepackage{tikz}
\usetikzlibrary{plotmarks}
\usepackage{url} 
\graphicspath{{Fig/}}
\usepackage{bbm}
\usepackage{subcaption}
\usepackage{amssymb}
\usepackage[section]{placeins}
\usepackage{amsthm}
\usepackage{amsmath}
\usepackage{amsfonts}
\usepackage{multirow}
\usepackage{amsthm}

\newtheorem{theorem}{Theorem}
\newtheorem{lemma}{Lemma}
\newtheorem{proposition}{Proposition}

\theoremstyle{definition}

\newtheorem{example}{Example}

\newcommand{\E}{\mathbb{E}}

\newcommand{\vc}[3]{\overset{#2}{\underset{#3}{#1}}}
\newcommand{\LPP}{Z}
\newcommand{\LP}{\mathcal{Z}}
\newcommand{\1}{\mathbbm{1}}

\theoremstyle{definition}
\newtheorem{assumption}{Assumption}

\title{A new class of nonparametric tests for second-order stochastic dominance based on the Lorenz P-P plot}

\author[1,2]{Tommaso Lando}
\affil[1]{Department of Economics, University of Bergamo, Italy }

\author[1]{Sirio Legramanti}
\affil[2]{Department of Finance, V\v{S}B-TU Ostrava, Czech Republic}

\date{}

\begin{document}

%

\maketitle

\begin{abstract}
\noindent Given samples from two non-negative random variables, we propose a family of tests for the
null hypothesis that one random variable stochastically dominates the other at the second order. Test statistics are obtained as functionals of the difference between the identity and the Lorenz P-P plot, defined as the composition between the inverse unscaled Lorenz curve of
one distribution and the unscaled Lorenz curve of the other. We determine upper bounds for such test statistics under the null hypothesis and derive
their limit distribution, to be approximated via bootstrap procedures. We then establish the asymptotic validity
of the tests under relatively mild conditions and investigate finite sample properties through simulations. The results show that our testing approach can be a valid alternative to classic methods based on the difference of the integrals of the cumulative distribution functions, which require bounded support and struggle to detect departures from the null in some cases. 
\end{abstract}

%
%
%


\section{Introduction}

The theory of stochastic orders deals with the problem of comparing pairs of random variables, or the corresponding distributions, with respect to concepts such as size, variability (or riskiness), shape, aging, or combinations of these aspects. The main notion in this context is generally referred to as the \textit{usual stochastic order} or \textit{first-order stochastic dominance} (FSD), which expresses the concept of one random variable being \textit{stochastically larger} than the other \citep{shaked}. For this reason, FSD has important applications in all those fields in which ``more'' is preferable to ``less'', clearly including economics. However, FSD is a restrictive criterion, {and is rarely} satisfied in {real-world applications. This has pushed economic theorists to develop} finer concepts, which formed the theory of \textit{stochastic dominance} (SD), taking into account variability and shape, in addition to size \citep{hadar1969,hanoch1969,Whitmore,fishburn,muliere1989,wang1998}. In this regard, the most commonly used SD relation is the \textit{second-order SD} (SSD), expressing a preference for the random variable which is stochastically larger or at least less risky, therefore combining size and dispersion into a single preorder. This has applications in economics, finance, operations research, reliability, and many other fields in which decision-makers {typically} prefer larger or at least less uncertain outcomes. 

Given a pair of samples from two unknown random variables of interest, statistical methods may be employed to establish whether such variables are stochastically ordered. In particular, we focus on a major problem in nonparametric statistics, that is testing the null hypothesis of dominance versus the alternative of non-dominance. About SSD, several procedures are available in the literature, some of which are described in the book of \cite{whang2019}. {We will now} recall a few of these approaches. 
\cite{dd} proposed a test for SSD based on the distance between the integrals of the cumulative distribution functions (CDF). The problem with this test is that dominance is evaluated only on a fixed grid, which may lead to inconsistency. 
\cite{bd} employed a similar approach, combined with bootstrap methods, to formulate a class of tests that are consistent under the assumption that the distributions under analysis are supported on a compact interval. 
\cite{donaldhsu} leveraged a less conservative approach to determine critical values compared to \cite{bd}, avoiding the use of the least favourable configuration. 
Note that all the aforementioned papers deal more generally with finite-order SD, and then obtain SSD as a special case. 
Alternatively, other works focused on tests for the so-called \textit{Lorenz dominance}, which is a scale-free version of SSD that applies to non-negative random variables. 
{For example,} \cite{bdlorenz} proposed a class of consistent tests for {the} Lorenz dominance that {rely on} the distance between empirical Lorenz curves. In this case, supports may be unbounded. Critical values are determined by approximating the limit distribution of a stochastic upper bound of the test statistic, similar to \cite{bd}. \cite{sunbeare} {used} a different and less conservative bootstrap approach to improve the power of such tests, {and} {established} asymptotic properties under less restrictive distributional assumptions.

The main idea of this paper follows from noticing that some stochastic orders, including FSD, can be expressed and tested via the classic P-P plot, also referred to as the \textit{ordinal dominance curve} \citep{hsieh,schmid,davidov,beare2015,tang,bearenew}. Following a similar approach, we propose a new class of nonparametric tests for SSD between non-negative random variables, in which the test statistic is based on what we refer to as the \textit{Lorenz P-P plot} (LPP), a kind of P-P plot based on \textit{unscaled Lorenz curves}. More precisely, the LPP is obtained as the functional composition of the inverse unscaled Lorenz curve of one distribution and the unscaled Lorenz curve of the other. The key property of the LPP is that, under SSD, it does not exceed the identity function on the unit interval. Therefore, the LPP stands out as a promising tool for detecting deviations from the null hypothesis of SSD. Namely, any functional that quantifies the positive part of the difference between the identity and the LPP can be used to construct a test statistic. This gives rise to a whole class of tests, depending on the choice of the functional. The $p$-values of such tests can then be computed via bootstrap procedures. In particular, we use a similar idea as in \cite{bdlorenz} to asymptotically bound the size of the test, and establish its consistency via the functional delta method. Note that the consistency of our family of tests is established without requiring a bounded support, which represents an advantage compared to classic methods based on integrals of CDFs. Moreover, our simulation studies show that our tests are often more reliable than the established KSB3 test by \cite{bd}, which may have problems detecting violations of the null hypothesis in some cases.

The LPP may also be used to define families of fractional-degree orders that are ``between'' FSD and SSD (or beyond SSD) by using a simple transformation. 
In this regard, we propose a method to define a continuum of SD relations, called \textit{transformed SD}, in the spirit of the recent works by \cite{muller2017}, \cite{lando}, and \cite{huang2020}. 
Interestingly, our tests can be easily adapted to this {more general} family of orders by simply transforming the samples through the same transformation used in the definition of transformed SD. 
In particular, FSD can be obtained as a limiting case, in which the empirical LPP of the transformed sample tends to the classic empirical P-P plot. This opens up the possibility of applying our class of tests to a wide family of stochastic orders.

The paper is organised as follows. Section~\ref{sect 2} introduces the LPP and describes the idea behind the proposed family of tests. In Section~\ref{sect 3}, we propose an estimator of the LPP and study its properties. The empirical process associated with the LPP is investigated in Section~\ref{sect 4}, where we establish a weak convergence result that can be used to derive asymptotic properties of the tests. Namely, in Section~\ref{sect 5}, we establish bounded size under the null hypothesis and consistency under the alternative one, for {both} independent {and} paired samples. The extension to a family of fractional-degree {orders} is discussed in Section~\ref{sect fraction}. In Section~\ref{sect sim}, we illustrate the finite sample properties of the tests through simulation {studies, focusing} on tests arising from sup-norm and  integral-based functionals. Finally, Section~\ref{sec_conclusion} contains our concluding remarks. All the tables and proofs are reported in the Appendix.

\section{Preliminaries}\label{sect 2}

Throughout this paper, $H$ denotes a general CDF supported on the non-negative half line, {with finite mean $\mu_H$}. In particular, we consider a pair of non-negative random variables $X$ and $Y$ with CDFs $F$ and $G$, respectively, and finite expectations. When $F$ and $G$ are absolutely continuous, we will denote their densities with $f$ and $g$, respectively. Given that stochastic orders depend only on distribution functions, for any order relation $\succ$ we may write $X\succ Y$ or $F\succ G$ interchangeably. 

Let $L^p(0,1)$, for $p\geq 1,$ be the class of real-valued functions on the unit interval equipped with the $L^p$ norm $||.||_p$, that is, for $v\in L^p(0,1),$ $||v||_p=(\int_0^1|v(t)|^pdt)^{1/p}$, and $L^\infty(0,1)$ be the class of bounded real-valued functions equipped with the uniform norm $||.||_\infty$. Moreover, let $C[0,1]$ be the space of continuous real-valued functions on [0,1] also equipped with the uniform norm. Henceforth, ``increasing'' means ``non-decreasing'' and ``decreasing'' means ``non-increasing''. Given a function~$r$, we denote with $r_+=\max(0,r)$ its positive part. If $r$ is increasing, $r^{-1}(y)=\inf\{x:r(x)> y\}$ denotes its right-continuos generalised inverse. Finally, $ \rightsquigarrow $ denotes weak convergence, while $\rightarrow_p$ denotes convergence in probability.

\subsection{Stochastic dominance}

We say that $X$ is larger than $Y$ with respect to FSD, denoted as $X\geq_1 Y$, if $F(x)\leq G(x),\forall x$. Equivalently, $X\geq_1Y$ if and only if $\E u(X)\geq \E u(Y)$ for any increasing function $u$. 
Within an economic framework, coherently with the expected-utility approach, one may assume that $X$ and $Y$ represent monetary lotteries and $u$ is a utility function. Under this perspective, FSD represents all non-satiable decision-makers, that is all those with an increasing utility, {and} therefore can be seen as one of the strongest ordering principles. On the other hand, FSD has a limited range of applicability since, in real-world applications, CDFs often cross and hence distributions cannot be ordered using this criterion. 

For this reason, weaker ordering relations have been introduced, among which the most important is the SSD. We say that $X$ is larger than $Y$ with respect to SSD, denoted as $X\geq_2 Y$, if $\int_{-\infty}^xF(t)dt\leq \int_{-\infty}^xG(t)dt,\forall x$. Equivalently, $X\geq_2Y$ if and only if $\E u(X)\geq \E u(Y)$ for any increasing and concave function~$u$. In economics, SSD generally represents all non-satiable and risk-averse decision-makers, expressing a preference for the random variable with larger values or smaller dispersion. For example, $X\geq_2 Y$ entails that $\E X\geq\E Y$ and, in case of equality, $\text{Var}(X)\leq \text{Var}(Y)$ and $\Gamma(X)\leq \Gamma(Y)$, where $\Gamma$ denotes the Gini coefficient. 
The above definitions may be generalised to $k$-th order SD, denoted as $X\geq_k Y$, $k=1,2,3,...$, and represented by the following integral inequality $F^{[k]}(x)\leq G^{[k]}(x),\forall x$, where {$H^{[1]}=H$ and $H^{[k]}(x)=\int_{-\infty}^xH^{[k-1]}(t)dt$, for $k\geq 2$. 
	
	Besides the classic definitions of SD discussed above, different notions --- often including FSD and SSD as special or limiting cases --- have been studied in the literature.
	Notable examples are the \textit{inverse} SD \citep{muliere1989}, which is based on recursive integration of the quantile function instead of the CDF, and coincides with classic SD at degrees 1 and 2, and also some fractional-degree SD relations that interpolate FSD and SSD \citep[see, e.g.,][]{muller2017}, as discussed in more detail in Section~\ref{sect fraction}.

	\subsection{The Lorenz P-P plot}
	The goal of this paper is to test the null hypothesis $\mathcal{H}_0:X\geq_2 Y$ versus the alternative $\mathcal{H}_1:X\not\geq_2 Y$. This requires estimating some kind of distance between the situation of dominance and the situation of non-dominance. The classic solution \citep{dd,bd} is to construct test statistics based on an empirical version of the difference $\int_{-\infty}^xF(t)dt-\int_{-\infty}^xG(t)dt$, which is expected to be large, at least at some point, if $\mathcal{H}_0$ is false. However, the main issue with the usual definition of SSD, based on these integrated CDFs, is that these are unbounded in $[0,\infty)$, so there are no uniformly consistent estimators for such integrals unless both distributions have bounded support. Not by chance, \cite{bd} require that  $F$ and $G$ have common bounded support $[0,a]$, with $a$ finite, to derive consistent tests of stochastic dominance of order $k$, including SSD. 
	To avoid this limitation, we rely on an alternative but equivalent definition of SSD in terms of the unscaled Lorenz curve, which is always bounded. In particular, we observe that some stochastic orders may be alternatively expressed in terms of a Q-Q plot \citep{landotransform} or a P-P plot \citep{lehmann}. Similarly, SSD can be characterised by the modified P-P plot described below.
	
	Let $H^{-1}
	$ be the (left-continuous) quantile function of the CDF $H$. 
	The unscaled Lorenz curve of $H$ is defined as $	L_H(p)=\int_0^pH^{-1}(t)dt,p\in[0,1]$.
	The symbol $L_H$ is often used for the scaled version of the Lorenz curve, that is $L_H/\mu_H$, while we use $L_H$ to denote the \textit{unscaled} Lorenz curve, for the sake of simplicity. Also note that $L_H:[0,1]\rightarrow [0,\mu_H]$ is increasing, convex and continuous in the unit interval. However, for technical reasons, we also let $L_H(p)=+\infty$ for $p>1$. With this extension, the generalised inverse function $L_H^{-1}:[0,\infty)\rightarrow[0,1]$ is increasing, concave, and continuous in $[0,\mu_H]$, while $L_H^{-1}(y)=1$ for $y>\mu_H$. 
	
	Now, given the pair of CDFs $F$ and $G$, consider the increasing continuous function 
	\begin{equation*}
		\LPP(p) =L_G^{-1}\circ L_F(p), \qquad p\in[0,1],
	\end{equation*}
	which takes values in $ [0,1\land L_G^{-1}(\mu_F)] $, where $ x \land y $ denotes the minimum between two real numbers $ x $ and~$ y $.
	Letting $\nu=1\land L_F^{-1}(\mu_G)$, note that, if $\mu_G<\mu_F$, then $\nu<1$ and $\LPP (p)=1$ for $p\in(\nu,1]$. 
	Given some point $y=L_F(p)$, for $p\in[0,1]$, the graph of $\LPP $ is a P-P plot with coordinates $(L^{-1}_F(y),L^{-1}_G(y))$, which will be referred to as the Lorenz P-P plot (LPP). Within an economic framework, $\LPP (p)$ returns the probability given by $ G $ to the average level of income corresponding to $L_F(p)$. In particular, if such a level cannot be reached under $G$, we have $\LPP (p)=1$. The LPP is scale-free, like the classic P-P plot; in particular, if $X$ and $Y$ are multiplied by a positive scale factor, then $\LPP$ remains unchanged. 
	
	To see how $\LPP$ can be leveraged to characterize SSD, first recall that $
	X\geq_2 Y$ if and only if $L_F(p)\geq L_G(p), \ \forall p\in[0,1],$ see, e.g., \citet[Ch. 4]{shaked}.
	Such a relation can be equivalently expressed in terms of~$\LPP $: 
	\begin{equation}\label{comp}
		X\geq_2 Y\iff \LPP (p)\geq p, \ \forall p\in[0,1].
	\end{equation}
	{It is generally complicated to obtain an explicit expression of $\LPP$ for parametric probabilistic models. Explicit calculations for the case of a Weibull versus a unit exponential distribution are provided in Example~\ref{ex weibull} below, while a graphical illustration is given in Figure~\ref{f weibull}. Differently, and more importantly for our testing purposes, the LPP can be computed quite easily in the empirical case, as discussed in Section~\ref{sect 3}.
		
		\begin{figure}[h]
			\centering
			\includegraphics[width=.5\textwidth]{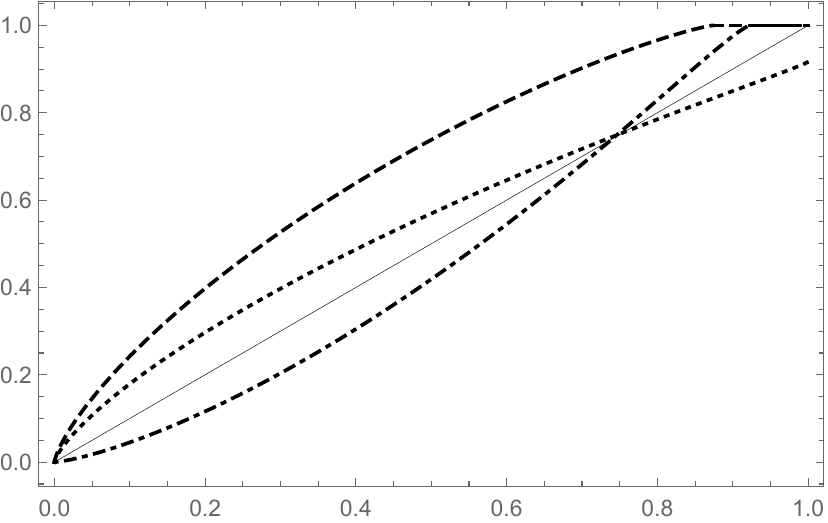}
			\caption{The LPP in Example~\ref{ex weibull} for: $a{=}2$, $b{=}1.5$ (dashed); $a{=}2$, $b{=}0.8$ (dotted); $a{=}0.6$, $b{=}1.2$ (dot-dashed).
				\label{f weibull}}
		\end{figure}
		
		\begin{example}\label{ex weibull}
			Consider the Weibull distribution $F(x)=1-\exp(-(x/b)^a),$ {with} $a,b>0$, and the unit exponential distribution, $G(x)=1-\exp(-x)$, both supported on $x\geq0$. 
			In this case, the LPP has the following expression (see the Appendix for details):$$\LPP(p)=1\land \mathcal{R}\left[1-\exp\left({1+W_{-1}\left(\frac{b \left(\Gamma \left(1+{1}/{a}\right)-\Gamma
					\left(1+{1}/{a},-\log (1-p)\right)\right)-1}{e}\right)}\right)\right],$$
			where $\mathcal{R}$ indicates the real part of a complex number, $\Gamma(\cdot,x)$ is the incomplete gamma function and $W_{-1}$ is the Lambert  function \citep{lambert}. Using the properties of SSD and the crossing conditions described in \cite{shaked}, it is easy to verify that $F\geq_2 G$ if and only if $a\geq1$ and $\mu_F=b \ \Gamma
			\left(1+{1}/{a}\right)\geq\mu_G=1$. Figure~\ref{f weibull} shows the behaviour of $\LPP$ when $F\geq_2 G$ and $F\not\geq_2 G$. 
		\end{example}

		\subsection{Detecting deviations from SSD}
		\label{construction}
		
		Denote the identity function by $I$. The representation of SSD in~\eqref{comp} can be leveraged to construct a test. In fact, $\mathcal{H}_0:X\geq_2 Y$ is false if and only if $I-\LPP$ is strictly positive at some point in the unit interval. Accordingly, departures from SSD can be detected by quantifying the positive part of the difference between $I$ and $\LPP$. This may be represented by some functional $\mathcal{T}$ applied to the difference $I-\LPP{\in C[0,1]}$, with $\mathcal{T}$ satisfying some desirable properties. 
		In particular, we propose a family of test statistics obtained as empirical versions of the functionals
		$$\mathcal{T}_p(I-\LPP)=||(I-\LPP)_+||_p,$$
		for $p\geq 1$, including $p=\infty.$ 
		
		It can be shown that functionals of this type satisfy the following properties.

		\begin{proposition}\label{properties}For every $v_1,v_2\in C[0,1]$ and for every $p\geq1$, 
			\begin{enumerate}
				\item If $v_1(x)=0,\forall x\in[0,1]$, then $\mathcal{T}_p(v_1)=0$ ;
				\item if $v_1(x)\leq0,\forall x\in[0,1]$, then ${\mathcal{T}_p}(v_2)\leq{\mathcal{T}_p}(v_2-v_1), \forall v_2$;
				\item if $v_1(x)>0$ for some $x\in[0,1]$, then $\mathcal{T}_p(v_1)>0$;
				\item $|\mathcal{T}_p(v_1)-\mathcal{T}_p(v_2)|\leq ||v_1-v_2||_\infty$;
				\item $c\mathcal{T}_p(v_1)=\mathcal{T}_p(cv_1)$, for any positive constant $c>0$;
				\item $\mathcal{T}_p$ is convex.
				\item For any $p_2\geq p_1\geq 1$, $\mathcal{T}_{p_2}(v_1)\geq \mathcal{T}_{p_1}(v_1)$.
			\end{enumerate}
		\end{proposition}

		Henceforth, we will denote simply by $\mathcal{T}$ any general functional satisfying the above properties~1)--6). These properties determine a family of functionals which may be used to obtain consistent tests. In particular, properties~2) and 3) completely characterise SSD, in that $\mathcal{T}({I}-\LPP)=0$ if and only if $X\geq_2 Y$, while $\mathcal{T}({I}-\LPP)>0$ if and only if $X\not\geq_2 Y$. Differently, property~7) deals just with the class $\mathcal{T}_p$ and shows that functionals of this kind measure the deviations from $\mathcal{H}_0$ in a monotone way, that is, smaller (larger) values of $p$ downsize (emphasize) deviations, represented by the function $(I-\LPP)_+$.
		Proposition~\ref{properties} generalises Lemma~2 of \cite{bdlorenz}, which deals with the special cases of $\mathcal{T}_1$ and $\mathcal{T}_\infty$. They introduced tests for the Lorenz dominance by applying $\mathcal{T}$ to the difference between the (scaled) Lorenz curves, that is $\mathcal{T}(L_G/\mu_G-L_F/\mu_F)$. One may extend their approach to SSD by considering  $\mathcal{T}(L_G-L_F)$ \citep[see, e.g.,][]{zhuang}. However, in this paper, we propose leveraging $\mathcal{T}(I-\LPP) $, which has some advantages over $\mathcal{T}(L_G-L_F)$. 
		For instance, $I-\LPP$ is scale-free by properties of the LPP. On the contrary, if $X$ and $Y$ are multiplied by a positive scale factor $c>0$, then the difference between the unscaled Lorenz curves becomes $c(L_G-L_F)$. Moreover, $|L_G-L_F|<\max(\mu_F,\mu_G)$ whereas $|I-\LPP|<1$.

		\section{Estimation of the LPP}\label{sect 3}
		
		\subsection{Sampling assumptions}
		Let $\mathcal{X}=\{X_1,...,X_n\}$ 
		and $\mathcal{Y}=\{Y_1,...,Y_m\}$ be i.i.d. random samples from $ F $ and $ G $, respectively.
		As in \cite{bdlorenz}, we will deal with two different sampling schemes: independent sampling and matched pairs. In the first scheme, the two samples $\mathcal{X}$ and $\mathcal{Y}$ are independent of each other, and sample sizes $n$ and $m$ may differ. In contrast, in the matched-pairs scheme, $n=m$ and we have $n$ i.i.d. pairs $\{(X_1,Y_1),...,(X_n,Y_n)\}$ drawn from a bivariate distribution with $F$ and $G$ as marginal CDFs.
		For both sampling schemes, we will consider the asymptotic regime in which $ n \to \infty $, $\lim_{n\rightarrow \infty} nm/(n+m)=\infty$ and $\lim_{n\rightarrow \infty} n/(n+m)=\lambda\in[0,1]$. These assumptions are quite standard in the literature \citep[see, e.g.,][]{bd} and imply that, as $ n $ diverges, $ m $ also goes to infinity with the same order. 
		
		\subsection{Empirical LPP}
		The abovementioned random samples $\mathcal{X}$ and $\mathcal{Y}$ yield the empirical CDFs $$F_n(x)=(1/n)\sum\nolimits_{i=1}^n\1(X_i\leq x)\quad\text{and }\quad G_m(x)=(1/m)\sum\nolimits_{i=1}^m\1(Y_i\leq x),$$respectively. We denote the order statistics of rank $k$ from $\mathcal{X}$ and $\mathcal{Y}$ with $X_{(k)}$ and $Y_{(k)}$, and their sample means with $\overline{X}_n$ and $\overline{Y}_m$, respectively. 
		By the plugin method, the empirical counterparts of $L_F$ and $L_G^{-1}$ are 
		$L_{F_n}$ and $L^{-1}_{G_m}$, where $L_{F_n}(p)=\int_0^{p}F_n^{-1}(t)dt$ for $p\in[0,1]$, $L_{G_m}$ is defined similarly, and $L_{G_m}^{-1}$ is the inverse of $L_{G_m}$. 
		Coherently with our definition of $L^{-1}_{G}$, we let $L_{G_m}^{-1}(p)=1$ for $p>\overline{Y}_m$. Note that $L_{F_n}$ coincides with the empirical unscaled Lorenz curve \citep{shorrocks1983}, that is a piecewise linear function joining the points $(k/n,(1/n)\sum_{i=0}^{k}X_{(i)})$, $k=0,...,n$, with $X_{(0)}:=0$. 
		
		Our definitions of $L_{F_n}$ and $L^{-1}_{G_m}$ differ from the ones, based on step functions, in \cite{csorgo2013}:
		$$\widetilde{L}_{F_n}(p)=
		\begin{cases}
			(1/n)\sum_{i=1}^{[np]+1}X_{(i)} \quad &p\in[0,1),\\
			\overline{X}_n &p=1,\\
			+\infty &p>1,
		\end{cases} $$
		$$ \widetilde{L}^{-1}_{G_m}(p)=\begin{cases}
			0 \qquad &p\in[0,Y_{(1)}/m]\\
			(k-1)/m &p\in[(1/m)\sum_{i=1}^{k-1}Y_{(i)},(1/m)\sum_{i=1}^{k}Y_{(i)}), \quad 2\leq k\leq m,\\
			1 &p\geq \overline{Y}_m\end{cases}$$
		where clearly $\widetilde{L}_{G_m}^{-1}(p)=\inf\{u:\widetilde{L}_{G_m}(u)>p\}$. Note that $L_{F_n}$ and $\widetilde{L}_{F_n}$ coincide at points $X_{(i)}$, $i=1,..,n$, and, likewise, $L^{-1}_{G_m}$ and $\widetilde{L}^{-1}_{G_m}$ coincide at points $i/n$, so these alternative empirical versions of $L_F$ and $L_G^{-1}$ are clearly asymptotically equivalent.
		
		According to the different empirical versions of $L_F$ and $L_G^{-1}$, we may obtain different estimators of~$\LPP $. 
		One may consider ${\LPP }_{n,m}={L}^{-1}_{G_m}\circ {L}_{F_n}$, which is a continuous piecewise linear function, or alternatively $\widetilde{\LPP }_{n,m}=\widetilde{L}^{-1}_{G_m}\circ \widetilde{L}_{F_n}$, which is a step function with jumps in the points $i/n$ $(i=1,...,n)$, taking values in $\{j/m: j=1,...,m \}$. In an economic framework, $\widetilde{L}_{G_m}^{-1}$ gives the relative frequency of observations from~$Y$ whose level of income is at most $(1/m)\sum_{j=1}^kY_{(j)}$. Therefore, $\widetilde{\LPP }_{n,m}(k/n)$ returns the relative frequency of observations from $Y$ whose level of income is at most $(1/n)\sum_{i=1}^kX_{(i)}$. Note that the value of ${\LPP }_{n,m}$ at its ``node'' points $ i/n $ does not generally coincide with the value of $\widetilde{\LPP }_{n,m}$ at its jump points. In this paper, we will use $\widetilde{\LPP }_{n,m}$ or ${\LPP }_{n,m}$ as is more convenient, since the two are asymptotically equivalent. In fact, the sup-distance among $\widetilde{\LPP }_{n,m}$ and $\LPP _{n,m}$ tends to zero as $n$ and $ m $ diverge, as established in the following proposition.
		\begin{proposition}\label{P1}
			For any $ n,m>0 $, $\sup|\widetilde{\LPP }_{n,m}-\LPP _{n,m}|\leq 1/m$. Hence, as $ m \to \infty $,	
			$\sup|\widetilde{\LPP }_{n,m}-\LPP _{n,m}| \to 0$.
		\end{proposition}
		In our asymptotic scenario, when $ n \to \infty $ we also have $ m \to \infty $, hence the second part of Proposition~\ref{P1} holds. Moreover, based on the strong uniform consistency of Lorenz curve estimators and their inverse functions \citep{goldie,csorgo2013}, we can prove the strong uniform consistency of $\LPP _{n,m}$ and~$\widetilde{\LPP }_{n,m}$.
		\begin{proposition}\label{P2}
			As $ n ,m\to \infty $,
			$\LPP _{n,m}\rightarrow \LPP $ and $\widetilde{\LPP }_{n,m}\rightarrow \LPP $ a.s. and uniformly in~$[0,1]$.
		\end{proposition}

		\section{Weak convergence of the LPP process}
		\label{sect 4}
		
		The empirical process associated with $\LPP$, henceforth referred to as the LPP process, may be useful to characterize the limit distribution of the test statistic under the null hypothesis of SSD. In this section, we study the asymptotic properties of such a process.
		Define the LPP process as 
		$$\LP _n(p)=\sqrt{r_n}(\LPP _{n,m}(p)-\LPP (p)), \qquad p\in[0,1],$$
		where ${r_n}={nm}/{(n+m)}$, and let $\nu_n=1\land L_{F_n}^{-1}(\overline{Y}_m)$ be the empirical counterpart of $\nu=1\land L_F^{-1}(\mu_G)$.
		For $\nu<1$ we know that $\LPP (t)=1$ when $t\in(\nu,1]$. In this case we have $\sup|\LP _n\1{(\nu,1]}|\rightarrow 0$ a.s., since also $\LPP _{n,m}(p)=1$ for $t\in(\nu_n,1]$ and $\nu_n\rightarrow \nu$ a.s. In other words, the interval $(\nu_n,1]$ contains no information. Accordingly, we are particularly interested in the asymptotic behaviour of $\LP _n$ restricted to $[0,\nu_n]$, namely $\LP _n\1{[0,\nu_n]}$. Weak convergence of the LPP process can be derived under the following assumptions.
		\begin{assumption} \label{ass1}
			Both $ F $ and $ G $ are continuously differentiable with strictly positive density, and have a finite moment of order $2+\epsilon$ for some $\epsilon>0$. Moreover, $F(0)=G(0)=0$.
		\end{assumption}
		\begin{assumption} \label{ass2}
			There exists some number $c>0$ such that $G^{-1}(0)=c$.
		\end{assumption}
		
		The latter assumption does not represent a limitation in terms of applicability. In fact, if $G^{-1}(0)=0$, one can apply the test to the shifted samples $\mathcal{X}+\epsilon$ and $\mathcal{Y}+\epsilon$, for some small $\epsilon>0$, recalling that $X\geq_2 Y$ if and only if $ X+\epsilon \geq_2 Y+\epsilon$. 
		In our simulations we set $\epsilon=10^{-4}$, obtaining results that are almost indistinguishable from those under $\epsilon=0$. However, since the unscaled Lorenz curve is not translation invariant, the outcome of any test based on it \citep[such as, e.g.,][]{andreoli,zhuang} may depend on the shift $\epsilon$. Actually, in our experiments, we noted that larger values of $\epsilon$ may even improve the power of the test.
		
		The following theorem establishes the weak convergence of $\LP_n$ under Assumptions~\ref{ass1}--\ref{ass2}, leveraging some recent results in \cite{kaji} that enable the derivation of the Hadamard differentiability of the map from CDFs to quantile functions. As discussed in \citet[][Section~2.4]{sunbeare}, this extends the applicability of earlier Hadamard differentiability conditions, based on stronger distributional assumptions such as bounded supports \citep[Lemma 3.9.23]{vw}. Then, the weak convergence of $\LP_n$ follows by the functional delta method \citep[Sect. 3.9]{vw}. 
		
		Let $\mathcal{B}$ be a centered Gaussian element of $C[0,1]\times C[0,1]$ with covariance function
		$Cov(\mathcal{B}(x_1,y_1),\mathcal{B}(x_2,y_2))=C(x_1\land x_2,y_1\land y_2)-C(x_1,y_1)C(x_2,y_2)$.
		Under the independent-sampling scheme, $C(x_1,y_1)=x_1y_1$ is the product copula, whereas, under the matched-pairs scheme, $C$ is the copula associated with the pair $(X_i,Y_i)$, $i=1,...,n$. Now, let $\mathcal{B}_1(x_1)=\mathcal{B}(x_1,1)$ and $\mathcal{B}_2(x_2)=\mathcal{B}(1,x_2)$. The random elements $\mathcal{B}_1$ and $\mathcal{B}_2$ are Brownian bridges that are independent under the independent-sampling scheme, but may be dependent under the matched-pairs one.
		
		\begin{theorem}\label{LPP}
			Under Assumptions~\ref{ass1}--\ref{ass2} and both independent-sampling and matched-pairs schemes, we have	$\sqrt{r_n}(\LPP _{n,m}-\LPP )\rightsquigarrow \LP \1{[0,\nu]}$ in $C[0,1],$ where 
			$$\LP = \frac{-\sqrt{\lambda}\int_0^\LPP \mathcal{B}_2(p)dG^{-1}(p)  +\sqrt{1-\lambda}\int_0^.\mathcal{B}_1dF^{-1}(p)}{G^{-1}\circ\LPP }.$$ 
		\end{theorem}
		It is interesting to observe that, if $X=_dY$, the result of Theorem~\ref{LPP} boils down to
		\begin{align*}
			\sqrt{r_n}(\LPP _{n,m}(t)-t) &\rightsquigarrow 
			\frac1{F^{-1}(t)}\int_0^t(\sqrt{\lambda}\mathcal{B}_2(u)-\sqrt{1-\lambda}\mathcal{B}_1(u))dF^{-1}(u)\\
			=\frac1{F^{-1}(t)}\int_0^t\widetilde{\mathcal{B}}(u)dF^{-1}(u) \quad \text{in } C[0,1],
		\end{align*}
		where $\widetilde{\mathcal{B}}$ is the Brownian bridge defined as
		$\widetilde{\mathcal{B}}=-\sqrt{1-\lambda}\mathcal{B}_1+\sqrt{\lambda}\mathcal{B}_2$.
		Finally, note that, by the asymptotic equivalence implied by Proposition~\ref{P1}, all the results in this section still hold if one replaces $\LPP _{n,m}$ with $\widetilde{\LPP }_{n,m}$.

		\section{Asymptotic properties of the test}\label{sect 5}
		As discussed in Section~\ref{construction}, deviations from $\mathcal{H}_0: X\geq_2 Y$ can be measured via the test statistic $\mathcal{T}_n=\sqrt{r_n} \ {\mathcal{T}} (I-\LPP _{n,m}).$ Intuitively, we reject $\mathcal{H}_0$ if $\mathcal{T}_n$ is large enough. However, since the null hypothesis is nonparametric, the main issue is how to determine the distribution of $\mathcal{T}_n$, or alternatively of an upper bound for $\mathcal{T}_n$, under~$\mathcal{H}_0$. Following the approach of \cite{bdlorenz}, it is easily seen that, under $\mathcal{H}_0$, the test statistic $\sqrt{r_n} \ \mathcal{T}(I-\LPP _{n,m})$ is dominated by $\sqrt{r_n} \ \mathcal{T}(\LPP -\LPP _{n,m})$, which therefore can be used to simulate $p$-values or critical values via bootstrap, thus ensuring that the size of the test is asymptotically bounded by some arbitrarily small probability $\alpha$. By the continuous mapping theorem, ${\sqrt{r_n}} \ \mathcal{T}(\LPP -\LPP _{n,m})$ is asymptotically distributed as $\mathcal{T}(\LP )$, allowing us to derive large-sample properties of the test. The limit behaviour of $\mathcal{T}_n$ under the null and the alternative hypotheses is established in the following lemma.

		\begin{lemma}\label{lemmatest}
			\begin{enumerate}\
				\item Under $\mathcal{H}_0$, $\sqrt{r_n} \ \mathcal{T}(I-\LPP _{n,m})\leq \sqrt{r_n} \ \mathcal{T}(\LPP -\LPP _{n,m})\rightsquigarrow \mathcal{T}(\LP )$. Moreover, for any $\alpha<1/2$, the $(1-\alpha)$ quantile of $\mathcal{T}(\LP )$ is positive, finite, and unique.
				\item Under $\mathcal{H}_1$, $\sqrt{r_n} \ \mathcal{T} (I-\LPP _{n,m})\rightarrow_p \infty$. 
			\end{enumerate}
		\end{lemma}
		From a practical point of view, the limit distribution of $\sqrt{r_n}\mathcal{T}(\LPP -\LPP _{n,m})$ under the null hypothesis may be approximated using a bootstrap approach, as discussed in the next subsection.
		
		\subsection{Bootstrap decision rule}
		Let us denote the bootstrap estimators of the empirical CDFs $F_n$ and $G_m$ as 
		$F^*_n$ and $G^*_m$, respectively:
		$$F^*_n(x)=(1/n)\sum\nolimits_{i=1}^n M_i^1\1(x\leq X_i), \qquad G^*_m(x)=(1/m)\sum\nolimits_{i=1}^m M_i^2\1(x\leq Y_i),$$
		where $M^1=(M_1^1,...,M^1_n)$ and $M^2=(M_1^2,...,M^2_m)$ are independent of the data and are drawn from a multinomial distribution according to the chosen sampling scheme. In particular, under the independent-sampling scheme, $M^1$ and $M^2$ are independently drawn from multinomial distributions with uniform probabilities over $n$ and $m$ trials, respectively. Under the matched-pairs scheme, we have $M^1=M^2$ drawn from the multinomial distribution with uniform probabilities over $n=m$ trials, which means that we sample (with replacement) pairs of data, from the $n$ pairs $\{(X_1,Y_1),...,(X_n,Y_n)\}$. 
		Correspondingly, by applying the definitions in Section~\ref{sect 2}, we obtain the bootstrap estimators of the unscaled Lorenz curves, denoted with $L_{F^*_n}$ and $L_{G^*_m}$, as well as the inverse $L^{-1}_{G^*_m}$, and we define 
		$\LPP ^*_{n,m}=L^{-1}_{G^*_m}\circ L_{F^*_n}$.
		As is shown below, the random process $\sqrt{r_n} \ \mathcal{T}(\LPP _{n,m}-\LPP ^*_{n,m})$ has the same limiting distribution as~$\mathcal{T}(\LP) $. Therefore, bootstrap $p$-values are determined by 
		$$p=P\{\sqrt{r_n} \ \mathcal{T}(\LPP _{n,m}-\LPP^* _{n,m})>\sqrt{r_n} \ \mathcal{T}(I-\LPP _{n,m})\},$$
		and can be approximated, based on $K$ bootstrap replicates, by
		$$p\approx (1/K) \sum\nolimits_{k=1}^K \1\{ \sqrt{r_n} \ \mathcal{T}(\LPP _{n,m}-\LPP ^*_{k;n,m})>\sqrt{r_n} \ \mathcal{T}(I-\LPP _{n,m})\},$$
		where $\LPP ^*_{k;n,m}$ is the $k$-th resampled realisation of $\LPP ^*_{n,m}$. As usual, the test rejects $\mathcal{H}_0$ if $p<\alpha$. 
		The asymptotic behaviour of the test is addressed by the following proposition.
		
		\begin{proposition}
			\label{ptest}
			Under Assumptions~\ref{ass1}--\ref{ass2} and the sampling schemes in Section~3.1,
			\begin{enumerate}
				\item If $\mathcal{H}_0$ is true, $\lim_{n\rightarrow\infty}P\{\text{reject }\mathcal{H}_0\}\leq\alpha$; 
				\item If $\mathcal{H}_1$ is true, $\lim_{n\rightarrow\infty}P\{\text{reject }\mathcal{H}_0\}=1$.
			\end{enumerate}
		\end{proposition}

		\section{Extension to fractional-degree SD}\label{sect fraction}
		
		An important topic in SD theory is represented by SD relations that are ``between'' FSD and SSD. This is motivated by the fact that FSD is a strong requirement, but, on the other hand, SSD corresponds to total risk-aversion, which is quite restrictive in some cases \citep{muller2017}. There are different ways to define classes of orders that interpolate between FSD and SSD, 
		and each leads to a different family of SD relations, typically parametrised by some real number that represents the strength of the dominance. The first attempt in this direction is ascribable to \cite{fishburn}, who used fractional-degree integration to interpolate the classic $k$-th order SD at all integer orders $k\geq1$. More recently, \cite{muller2017}, \cite{huang2020}, and \cite{lando} proposed different parametrizations, with different interpretations and properties, which coincide with classic SD only at orders 1 and 2. In this section, we introduce a simple but very general family of fractional-degree orders, which have the advantage that they can be easily tested using the LPP method discussed earlier. Such a family can be defined as follows. 
		
		Let $\mathcal{U}$ be the family of increasing absolutely continuous functions $ u $ over the non-negative half line. Under an economic perspective, $u$ may be understood as a {utility} function, assigning values to monetary outcomes. For some $u\in \mathcal{U}$, we say that $X$ dominates $Y$ with respect to $u$-\textit{transformed stochastic dominance} ($u$-TSD), and write $X\geq^T_u Y$, if $u(X)\geq_2 u(Y)$.
		TSD has been studied by \cite{meyer1977}, who denoted it as SSD with respect to $u$, and by \cite{huang2020}, who focused on a particular parametric choice of $u$. In fact, since $u$-TSD represents SSD between the transformed random variables $u(X)$ and $u(Y)$, then it can be simply expressed through the LPP of $u(X)$ and $u(Y)$. 
		
		The behaviour of TSD clearly depends on the choice of $u$. To understand this behaviour, let $u,\tilde{u}\in\mathcal{U}$ be two transformation functions defined on the same interval. Generalizing 
		\cite{chan1990}, we say that $u$ is \textit{more convex} than $\tilde{u}$ and write $u \geq_c \tilde{u}$ iff $u \circ \tilde{u}^{-1}$ is convex. The following theorem shows that TSD can be equivalently expressed in terms of expected utilities, thus generalizing Theorem 1 of \cite{huang2020}.
		\begin{theorem} \label{TT}
			$X\geq^T_{u} Y$ if and only if $\E(\phi(X))\geq\E(\phi(Y))$, for every increasing utility $\phi$ such that $u\geq_c \phi.$
		\end{theorem}
		
		It is easy to see that, if $u$ and $\phi$ are twice differentiable,  the condition $u\geq_c \phi$ is equivalent to $\rho_\phi(x)\geq \rho_u(x),\forall x$, where $\rho_g(x)={g''(x)}/{g'(x)}$ is the Arrow-Pratt index of absolute risk aversion associated with the utility function $g$.  Moreover, the following general properties hold.
		
		\begin{theorem} \label{T7}
			\begin{enumerate}\
				\item If $u_1\geq_c u_2$ then  $X\geq^T_{u_1} Y\implies X\geq^T_{u_2} Y$;
				\item $X\geq_1 Y$ if and only if $X\geq^T_{u} Y,\forall u\in\mathcal{U}$.
			\end{enumerate}
		\end{theorem}
		Intuitively, the degree of convexity of the function $u$ determines the strength of the SD relation, and SSD is obtained by taking $u$ to be the identity function, whereas FSD is obtained when $u$ is infinitely ``steep''.
		
		Families of utility functions within $\mathcal{U}$ can be obtained easily by composing the quantile function and the CDF of two absolutely continuous random variables. 
		For example, one may consider the class of utility functions studied by \cite{huang2020} and given by $u_c(x)=\exp{((1/c-1)x)},$ for $c\in(0,1)$. 
		Since this paper deals with tests for non-negative random variables, we focus on a simpler choice, that is $u_\theta(x)=x^{\theta}$, with $\theta\geq0$. Correspondingly, hereafter we denote the ordering relation $X\geq^T_{u_\theta} Y$ with $X\geq^T_{1+1/\theta} Y$, thus yielding a continuum of SD relations that get stronger and stronger as $\theta$ grows. By Theorem~\ref{T7}, this order is characterised by those utility functions that have an Arrow-Pratt index larger than or equal to ${(\theta-1)}/{x}$. 
		
		Since $X\geq^T_{1+1/\theta} Y$ is equivalent to $X^\theta \geq_{2} Y^\theta$, a test for $\mathcal{H}_0^{1+1/\theta}:X\geq^T_{1+1/\theta} Y$ versus $\mathcal{H}_1^{1+1/\theta}: X\not\geq^T_{1+1/\theta} Y$ is readily obtained by applying our method to the LPP of the transformed random samples $\mathcal{X}^\theta$ and $\mathcal{Y}^\theta$. 
		In particular, we consider the \textit{generalised LPP}, given by $\widetilde{\LPP}_{n,m}^\theta=(\widetilde{L}^\theta_{G_m})^{-1}\circ \widetilde{L}^\theta_{F_n}$, where $\widetilde{L}^\theta_{F_n}$ and $\widetilde{L}^\theta_{G_m}$ are the empirical (step-valued) unscaled Lorenz curves corresponding to the transformed samples $\{X_i^{\theta}: i=1,...,n\}$, and $\{ Y_j^{\theta}: j=1,...,m \}$. $\widetilde{\LPP}_{n,m}^\theta$ is a generalised P-P plot, in that it coincides with $\widetilde{\LPP}_{n,m}$ for $\theta=1$. More interestingly, we prove that, as $\theta\rightarrow\infty$,  $\widetilde{\LPP}_{n,m}^\theta$ tends to the classic P-P plot of the non-transformed samples, that is to $G_m\circ F_n^{-1}$, as depicted in Figure~\ref{f6}. In particular, one may always find some $\theta$ large enough such that the two P-P plots coincide, meaning that our tests may be also applied to FSD, expressed as $G\circ F^{-1}(x)\geq x$. (tests for FSD based on the P-P plot have been studied, e.g., by \cite{davidov} and \cite{bearenew}). In fact, this idea is coherent with the intuition that, for $\theta\rightarrow\infty$, the stochastic inequality $X\geq^T_{1+1/\theta} Y$ reduces to $X\geq_{1} Y$, expressed as $F(x)\leq G(x),\forall{x}$, as formally established in the following theorem.

		\newpage
		\begin{theorem} \label{T8}
			\begin{enumerate}\
				\item For $\theta\rightarrow\infty$, $X\geq_1 Y$ if and only if $X\geq^T_{1+1/\theta} Y$;
				\item There exists some $\theta_0$ such that, for $\theta>\theta_0$, the generalised LPP coincides with the classic P-P plot, that is, $\widetilde{\LPP }_{n,m}^\theta=G_m\circ F_n^{-1}$.
			\end{enumerate}
		\end{theorem}
		
			\begin{figure}
			\centering
			\includegraphics[width=.9\textwidth]{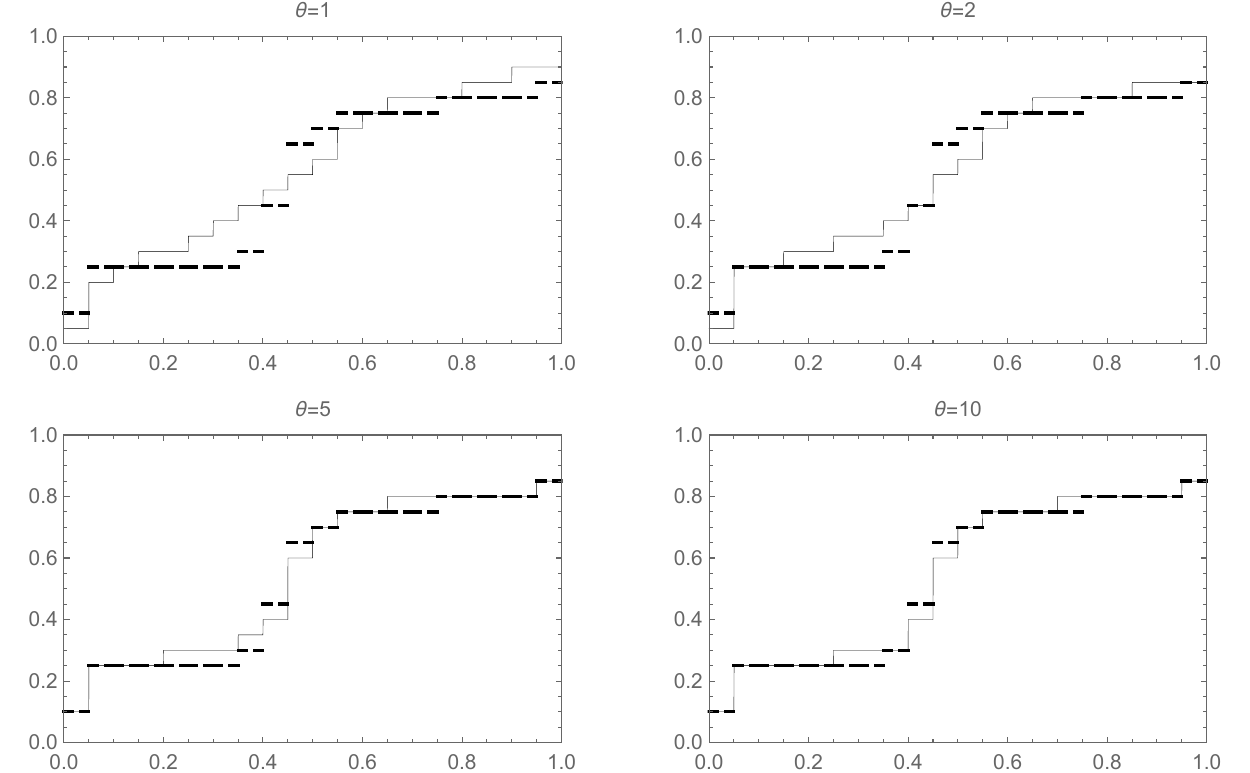}
			\caption{The P-P plot (dashed) of two samples of size $n=m=20$ versus the generalised LPP $\widetilde{\LPP }_{n,m}^\theta$ (solid), for $\theta=1,2,5,10$. In this example, the plots coincide for $\theta\geq48$.\label{f6}}
		\end{figure}
		
		To test FSD as a limit case of TSD, one should choose a value of $\theta$ that ensures the result above. However, if $\theta$ is too large, computations may be difficult, depending on the precision of the software used. We recommend using $\theta=50$, which corresponds to testing $\geq^T_{1.02}$, for a good approximation of FSD.

		\section{Simulations}\label{sect sim}
		We perform numerical analyses to investigate the finite-sample properties of the proposed tests. In all simulations, we consider a significance level $ \alpha=0.1 $, and run $ 500 $~experiments, with $ 500 $~bootstrap replicates for each experiment. For simplicity, we set $n=m$, so henceforth we will drop the subscript $m$. Namely, we consider $n=m=50,100,200,500,1000$. The shift $\epsilon$ is set to $ 10^{-4} $, as discussed in Section~\ref{sect 4}.		
		All computations have been performed in R, and the code is openly available at \url{https://github.com/siriolegramanti/SSD}.
		
		In light of Proposition~\ref{P1}, instead of ${\LPP }_{n}$ we use $\widetilde{\LPP }_{n}$, which can be computed faster. Accordingly, we consider two different test statistics, namely $\mathcal{T}_\infty(I-\widetilde{\LPP }_{n})$ and $\mathcal{T}_1(I-\widetilde{\LPP }_{n})$; see Section~\ref{construction}. For $n=m$, $\mathcal{T}_\infty$ and $\mathcal{T}_1$ can be rewritten, respectively, as
		$$\mathcal{T}_\infty(I-\widetilde{\LPP }_{n})=\max_i \left( \frac in - \widetilde{\LPP }_{n}\left( \frac in \right) \right), \qquad \mathcal{T}_1(I-\widetilde{\LPP }_{n})= \sqrt{n} \frac{1}{n} \sum\nolimits_{i=1}^n \Psi\left(\frac{2i-1}{2n}\right), $$
		where $ \Psi(t)=(t-\widetilde{\LPP }_{n}(t))_+ $.
		Our results are compared with those obtained from the tests of \cite{bd}, which represent the state of the art for SSD tests. In particular, \cite{bd} propose three bootstrap-based tests, based on a least favourable configuration, denoted as KSB1, KSB2, and KSB3, which differ just for the bootstrap method employed to simulate the $p$-values. We focus on KSB3 since it is based on the approach that is most similar to ours. Moreover, KSB3 seems to provide the best results compared to KSB1 and KSB2 as far as concerns SSD; see tables II-A and II-B of \cite{bd}. The $p$-values of KSB3 were computed using a grid of evenly spaced values $t_1<\ldots<t_r$, where $ t_1 $ and $t_r$ are the smallest and the largest values in the pooled sample, respectively. As for the number of grid points, we set $r=100$ as in \cite{bd}, but we did not notice substantial differences in increasing $ r $.
		
		Note that one pair of distributions gives rise to two different hypothesis tests. In fact, one may test $\mathcal{H}_0: F\geq_2 G$ versus $\mathcal{H}_1: F\not\geq_2 G$, but also the reverse hypothesis test, denoted as $\mathcal{H}^R_0: G\geq_2 F$ versus $\mathcal{H}^R_1: G\not\geq_2 F$. Except for the trivial case $F=G$, if $\LPP$ does not cross the identity we may have that $\mathcal{H}_0$ is true while $\mathcal{H}^R_0$ is false, or vice versa; differently, if $\LPP$ crosses the identity, $\mathcal{H}_0$ and $\mathcal{H}^R_0$ are both false.

		\subsection{Size properties}\label{sect size}
		To investigate the behaviour of the tests under the null hypothesis, we simulate samples from the Weibull family, denoted by $W(a,b)$, with CDF $F_W(x;a,b)=1-\exp\{-\left({x}/{b}\right)^a\}$. Since the mean of a $W(a,b)$ is $b/ q_a $, where $q_a=1/{\Gamma(1+ 1/a)}$, we let $F\sim W(a,q_a)$, for $a=1,1.25,1.5,1.75,2$, and fix $G\sim W(1,1)$. All these distributions have mean 1, and in all these cases $\mathcal{H}_0$ holds. Clearly, for $a=1$ we have $F=G$, whereas the dominance of $ F $ over $ G $ becomes stronger, and more apparent, for larger values of $a$. 
		
		The results in Tables~\ref{t1}, \ref{tw1l}, \ref{tw2l} and \ref{tw3l} confirm that the proposed tests, both with $\mathcal{T}_\infty$ and $\mathcal{T}_1$, behave as described in Proposition~\ref{ptest}, part~1. Namely, the rejection rate tends to be bounded by $\alpha=0.1$ under $\mathcal{H}_0$. More specifically, we observe that the rejection rate of the proposed tests tends to $\alpha$ when $F=G$ (see Table~\ref{t1}), while it tends to 0 when $ F $ strictly dominates $ G $ (see Tables~\ref{tw1l}, \ref{tw2l} and \ref{tw3l}). The rejection rate for the KSB3 test by \cite{bd} is also asymptotically bounded by $ \alpha $ but, when the dominance is stronger, it is still about $\alpha$ for $ n=1000 $. For such a sample size, the rejection rate of both the proposed tests has already reached 0.

		\subsection{Power properties}
		We now investigate the behaviour of the tests under $\mathcal{H}_1$. In particular, we focus on cases in which $F$ is dominated by $G$, so that $\mathcal{H}_0$ should be rejected quite easily since $\LPP$ is always below the identity. As we discuss in~\ref{sub wei}, the three tests considered behave quite similarly in such cases. We also focus on critical cases in which neither of the two distributions dominates the other, and therefore $\LPP$ crosses the identity. In particular, the most critical situation for our class of tests is when $\LPP$ is above the identity except for a small interval (see Figure~\ref{f4}). 
		The simulation results in~\ref{LNmix cases} and~\ref{SM cases} show that, in some of the most difficult cases, $\mathcal{T}_1$ and KSB3 struggle to reject $\mathcal{H}_0$, whereas the proposed $\mathcal{T}_\infty$ test stands out as the most reliable.
		
		\subsubsection{Weibull distribution}\label{sub wei}
		Using the same distributions as in Section~\ref{sect size}, except for the case $F=G$, we have that $F>_2 G$ (strictly) and therefore $G\not\geq_2 F$. In these cases, $\LPP$ is always above the identity. The results, reported in Tables~\ref{tw1r}, \ref{tw2r} and \ref{tw3r}, show that the power of the tests increases with the sample size. In particular, $\mathcal{T}_1$ seems to outperform $\mathcal{T}_\infty$ for smaller sample sizes, while both the proposed $\mathcal{T}_1$ and $\mathcal{T}_\infty$ tests provide larger power compared to the KSB3 test by \cite{bd}.

		\subsubsection{Lognormal mixture vs. lognormal distribution}
		\label{LNmix cases}
		
		As a more critical example, we focus on a special case considered by {\citet[][Case~5]{bd}. Here, $F$ is a mixture of lognormal distributions, namely $F=0.9F_{LN}(\cdot;0.85;0.4)+0.1F_{LN}(\cdot;0.4,0.4)$, 
			whereas $G=F_{LN}(\cdot;0.86,0.6)$. These CDFs cross multiple times, and also $\LPP$ crosses the identity from below so that $F\not\geq_2 G$ but also $G\not\geq_2 F$. In other words, both $\mathcal{H}_0$ and $\mathcal{H}^R_0$ are false. In the latter case, the null hypothesis is hard to reject, because $\LPP$ crosses the identity from above, and it exceeds the identity just in a small subset of the unit interval. Note that \cite{bd} just apply their test to $\mathcal{H}_0$ versus~$\mathcal{H}_1$, overlooking the reverse situation $\mathcal{H}^R_0$ versus $\mathcal{H}^R_1$. 
			As illustrated in Table~\ref{tln}, KSB3 seems to outperform our tests in detecting $\mathcal{H}_1$. In particular, $\mathcal{T}_1 $ exhibits quite a poor performance with the sample sizes considered (to increase its power up to 0.68, we need to reach $n=5000$). Conversely, KSB3 has a really poor performance in rejecting~$\mathcal{H}^R_0$, while the proposed $\mathcal{T}_\infty $ and $\mathcal{T}_1 $ tests provide a large power in this critical setting.
			
			\subsubsection{Singh-Maddala Distribution}
			\label{SM cases}	
			As a third case, let us consider the Singh-Maddala distribution, denoted as $\text{SM}(a,q,b)$, with CDF
			$ F_\text{SM}(x;a,q,b)=1-[1+(x/b)^a]^{-q}$.
			In all the following scenarios, the scale parameter $ b $ is set to 1 and hence omitted, while the two shape parameters $ a $ and $ q $ vary.
			As in Section~\ref{LNmix cases}, we generate scenarios in which $\LPP$ crosses the identity. 
			In particular, we target the worst-case scenarios for our proposed tests, to investigate their limitations, by setting $F\sim\text{SM}(1.5,q)$ and $G\sim\text{SM}(1,q)$, for $q=1.2,1.5,1.8$.
			As shown in Figure~\ref{f3}, larger values of $q$ correspond to cases in which it is harder to detect the difference between $\LPP$ and the identity, especially using $\mathcal{T}_1$. Tables~\ref{tsm1l}, \ref{tsm2l} and \ref{tsm3l} show that KSB3 delivers larger power compared to our tests in such critical cases. In particular, while the performance of $\mathcal{T}_\infty$ significantly improves for larger samples and lower $ q $, the power of $\mathcal{T}_1$ is constantly close to 0, even for $n=1000$ and $ q=1.2 $. In light of part~7) of Proposition~\ref{properties}, this is due to the fact that $\mathcal{T}_1$ downsizes the deviations from the null, which are hardly classified as ``large", at least with the sample sizes considered. Indeed, the $p$-values of $\mathcal{T}_1$ tend to decrease and to be less variable as $n$ grows, coherently with Proposition~\ref{ptest} (see Figure~\ref{f4}). 
			This suggests that the power of $\mathcal{T}_1$ may eventually tend to 1 for larger samples. 
			However, when applied to the reverse hypotheses $\mathcal{H}_0^R$ and $\mathcal{H}_1^R$, the proposed tests $\mathcal{T}_\infty$ and $\mathcal{T}_1$ exhibit good performance, with rejection rates significantly increasing with $n$; see Tables~\ref{tsm1r}, \ref{tsm2r} and \ref{tsm3r}. On the contrary, KSB3 struggles to detect non-dominance and its power remains close to~$ 0 $, even for large samples.

		\begin{figure}[t]
			\centering
			\includegraphics[width=.5\textwidth]{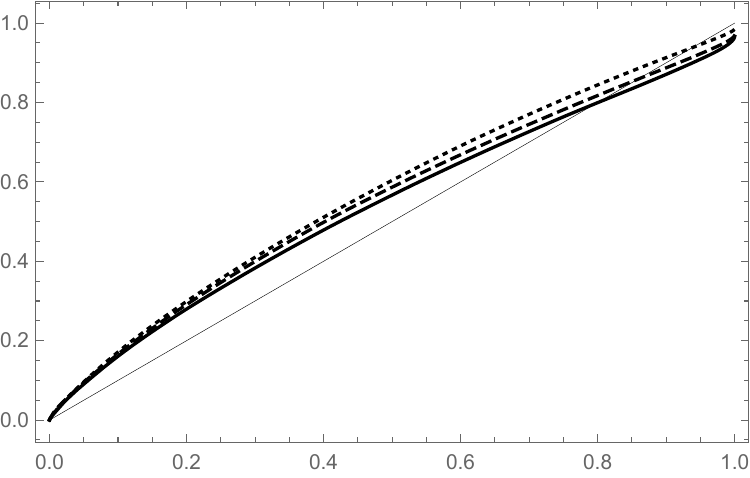}
			\caption{The behaviour of $\LPP$ for $q=1.2$ (solid), $q=1.5$ (dashed) and $q=1.8$ (dotted) in the Singh-Maddala case. Especially for $q=1.8$, it becomes very hard to detect deviations from $\mathcal{H}_0$. In the reverse cases, the LPPs are just the inverse functions of these.\label{f3}}
		\end{figure}

			\subsection{Paired samples}
			
			To simulate dependent samples we first draw a sample $\{(Z^1_i,Z^2_i):i=1,\ldots,n\}$ from a bivariate normal distribution, with standard marginals and correlation coefficient $\rho$. Then, by transforming the data via the standard normal CDF $\Phi$, we obtain a dependent sample from a bivariate distribution with uniform marginals $\{U_i^1=\Phi(Z^1_i): i=1,\ldots,n\}$ and $\{U_i^2=\Phi(Z^2_i): i=1,\ldots,n\}$. Finally, a dependent sample from a bivariate distribution with margins $F$ and $G$ is obtained as $\{(F^{-1}(U_i^1),G^{-1}(U_i^2)):i=1,\ldots,n\}$. In particular, we consider $\rho=0.25,0.5,0.75$. As in the previous subsections, we compare our results with those of KSB3. Note that, although \cite{bd} assume independence to prove the consistency properties of such a test, our simulations reveal that KSB3 exhibits a good performance even in the dependent case.
			
			In this paired setting, we consider the same Singh-Maddala distributions as in Section~\ref{SM cases}, focusing on the cases $q=1.2$ and $q=1.8$. The results, reported in Tables~\ref{tsm125}--\ref{tsm275}, confirm the ones reported in the previous tables, although it can be seen that a stronger dependence generally leads to larger rejection rates. Even $\mathcal{T}_1$, which was struggling with independent Singh-Maddala samples, shows a more evident decreasing trend in the empirical distribution of the $p$-values when samples are paired (see Figure~\ref{f5}). 
			
			\begin{figure}[h]
				\centering
				\begin{subfigure}[b]{0.49\textwidth}
					\centering
					\includegraphics[trim={0 1.5cm 0 1.5cm},clip,width=\textwidth]{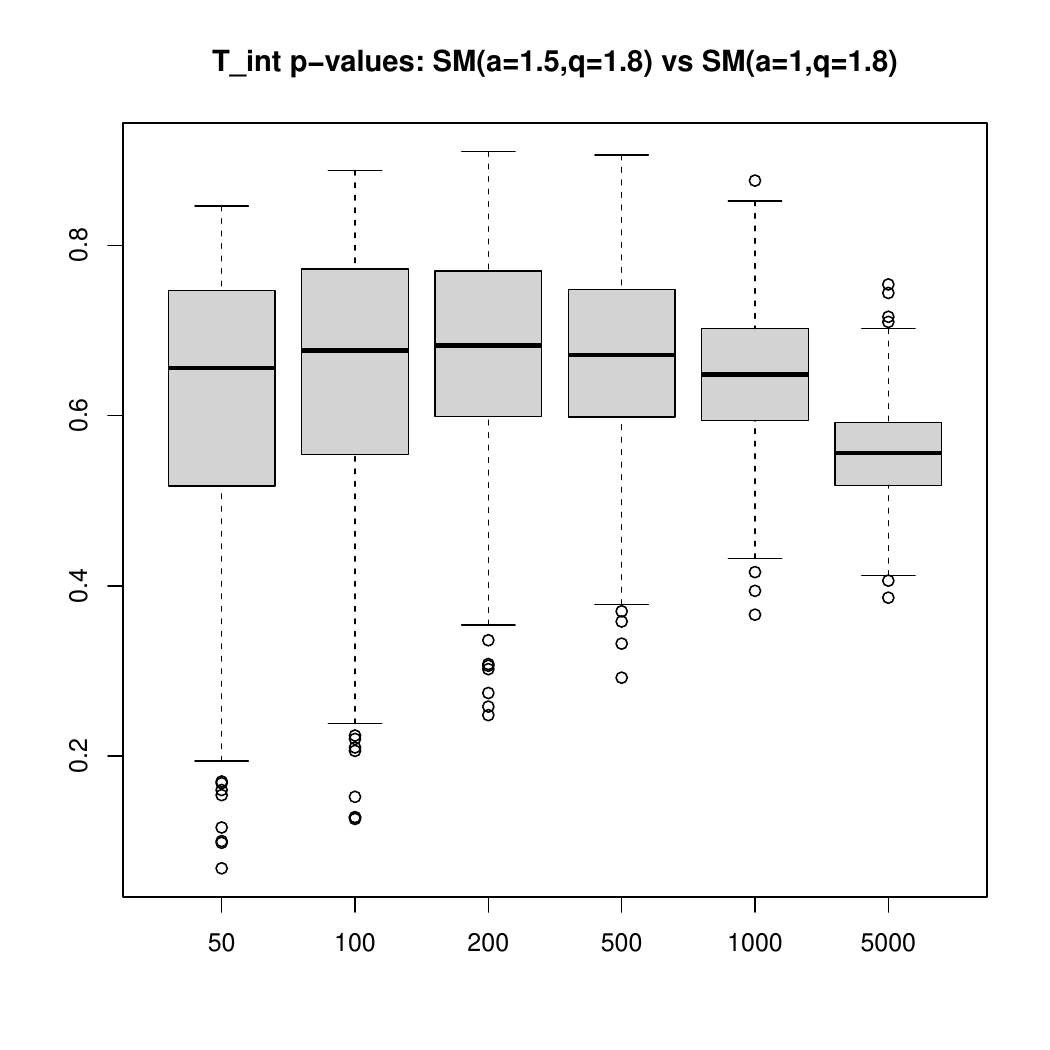}
					\caption{Independent Singh-Maddala ($q=1.8$)
						\label{f4}}
				\end{subfigure}
				\hfill
				\begin{subfigure}[b]{0.49\textwidth}
					\centering
					\includegraphics[trim={0 1.5cm 0 1.5cm},clip,width=\textwidth]{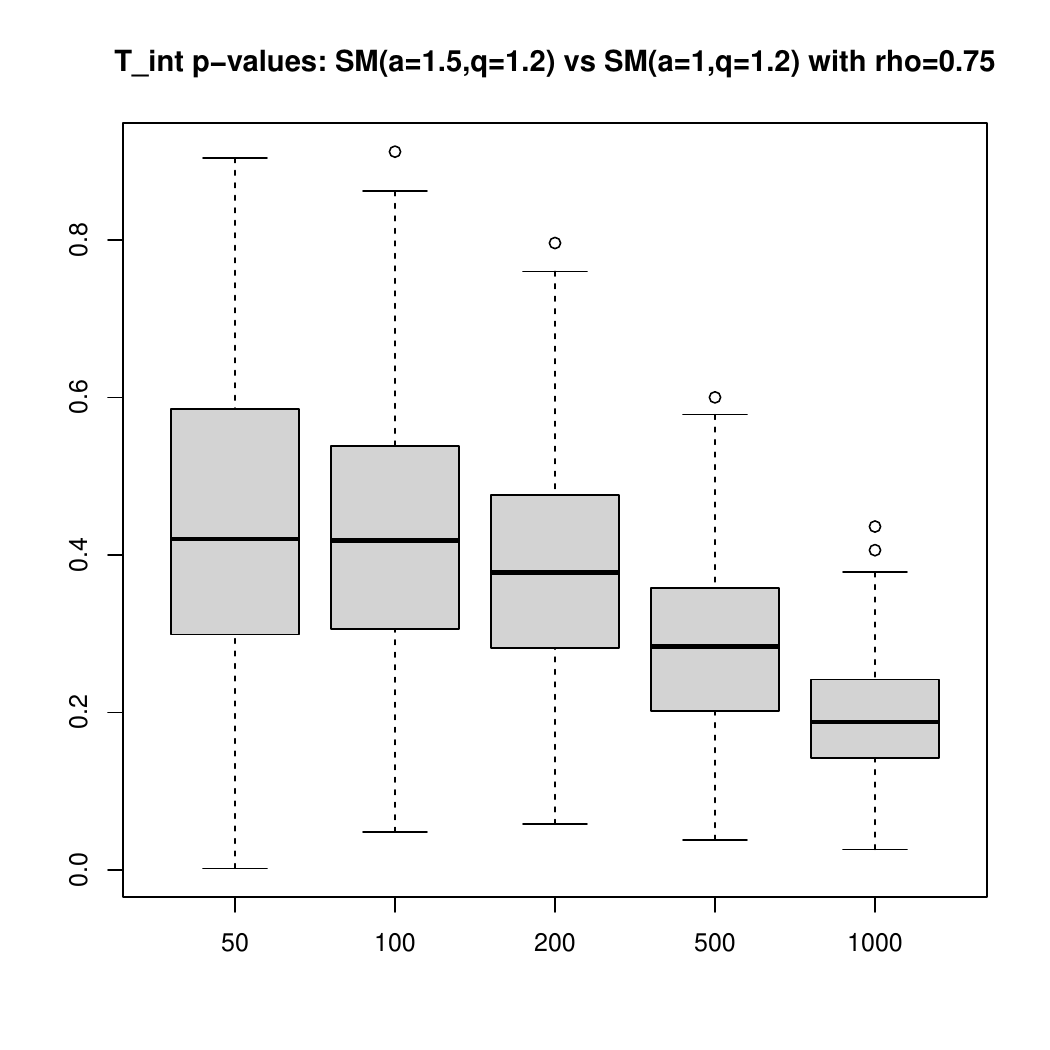}
					\caption{Dependent Singh-Maddala ($q=1.2$, $\rho=0.75$) \label{f5}}
				\end{subfigure}
				\caption{Box-plots of the simulated $p$-values of $\mathcal{T}_1$}
			\end{figure}
			
			
			%

			\subsection{Test for FSD}
			
			As discussed in Section~\ref{sect fraction}, our methodology also allows to test TSD, including an approximation of FSD, obtained as $\geq_{1+1/\theta}^T$ with $\theta\rightarrow\infty$.
			We then apply the method described in Section~\ref{sect fraction} to the same Singh-Maddala distributions studied in Section~\ref{SM cases}. 
			Since in these cases, SSD does not hold, we have that, \textit{a fortiori}, the FSD null hypothesis, denoted as $\mathcal{H}_0^{1}:F\geq_1 G$, is also false. This hypothesis can be tested using a sufficiently large value of $\theta$, as discussed in Section~\ref{sect fraction}. 
			In particular, we set $\theta=50$, which corresponds to approximating the FSD null hypothesis, $ \mathcal{H}_0^{1} $, with $\mathcal{H}_0^{1.02}$. Our method is compared with the FSD version of the KSB3 test described in \cite{bd}.
			In contrast to the KSB3 test for SSD, this latter test may be shown to be consistent even in the case when the distributions have unbounded supports.
			
			All the tests considered tend to provide a larger simulated power compared to the SSD case. This is logical since FSD is more stringent than SSD, and therefore, for the same pairs of distributions, it is easier to detect violations of FSD rather than of SSD. 
			The results in Tables~\ref{tfsd1}--\ref{tfsd3} show that KSB3 tends to provide larger power than our $\mathcal{T}_\infty$ and $\mathcal{T}_1$ tests under $\mathcal{H}^1_1:F\not\geq_1 G$. On the contrary, under the reverse alternative $(\mathcal{H}^1_1)^R:G\not\geq_1 F$, KSB3 exhibits a worse performance, also showing an unexpected behaviour, in that its rejection rates first increase and then decrease as $ n $ grows.

			\section{Concluding remarks}
			\label{sec_conclusion}
			
			In this paper, we proposed leveraging the LPP as a new tool to detect deviations from SSD in the case of non-negative random variables. 
			The same approach can be used to test TSD, hence including FSD as a limit case. The asymptotic properties in Section~\ref{sect 5} and the numerical results in Section~\ref{sect sim} show that our family of tests can be a valid alternative to the established tests based on the difference between integrals of CDFs, such as the tests in \cite{bd}. In particular, the KSB3 test is outperformed by our proposed sup-based test $\mathcal{T}_\infty$ in most of the cases analysed, sometimes with a remarkable gap. 
			
			Among the two tests proposed, our simulations reveal that the sup-based test $\mathcal{T}_\infty$ is also overall more reliable than the integral-based $\mathcal{T}_1$, which has lower power in the most critical cases. However, both tests may be useful. In fact, in light of Proposition~\ref{properties} part~7), and according to our numerical results in Section~\ref{SM cases} and Section~\ref{LNmix cases}, $\mathcal{T}_\infty$ performs better than $\mathcal{T}_1$ when deviations from $\mathcal{H}_0$ are subtle, while $\mathcal{T}_1$ provides higher power than $\mathcal{T}_\infty$ when deviations are more apparent. Therefore, in applications, it could be useful to use both tests and compare the $p$-values. It is also worth noting that our proposed tests seem to improve in terms of power when the samples are dependent. 
			
			In general, the advantage of using the LPP instead of integrals of CDFs is that it can be approximated uniformly, which allows to establish asymptotic properties without requiring bounded support; moreover, the LPP has a different sensitivity in detecting violations of SSD, compared to other methods. 
			Finally, the power of our tests may be improved further by combining the same proposed test statistics with different and less conservative bootstrap schemes. The latter represents an interesting direction for future work.

		\section*{Appendix A: Tables} \label{app_tables}
	
	\begin{table}[hbt!]
		\centering
		\begin{tabular}{rrrr}
			\hline
			$n$ & $\mathcal{T}_\infty$ & $\mathcal{T}_1$ & KSB3 \\ 
			\hline
			50 & 0.09 & 0.15 & 0.10 \\ 
			100 & 0.10 & 0.15 & 0.11 \\ 
			200 & 0.12 & 0.15 & 0.11 \\ 
			500 & 0.09 & 0.10 & 0.09 \\ 
			1000 & 0.08 & 0.09 & 0.09 \\ 
			\hline
		\end{tabular}
		\vspace{.1in}
		\caption{Rejection rates for $\mathcal{H}_0: F \geq_2 G$ under $F, G\sim W(1,1)$; independent samples}
		\label{t1}
	\end{table}
	
	\begin{table}[hbt!]
		\begin{subtable}[h]{0.5\textwidth}
			\centering
			\begin{tabular}{rrrr}
				\hline
				$n$ & $\mathcal{T}_\infty$ & $\mathcal{T}_1$ & KSB3 \\ 
				\hline
				50 & 0.08 & 0.10 & 0.12 \\ 
				100 & 0.04 & 0.06 & 0.11 \\ 
				200 & 0.03 & 0.02 & 0.10 \\ 
				500 & 0.02 & 0.00 & 0.11 \\ 
				1000 & 0.00 & 0.00 & 0.09 \\ 
				\hline
			\end{tabular}
			\caption{Test for $\mathcal{H}_0: F \geq_2 G$ (true)}
			\label{tw1l}   
		\end{subtable}
		\hfill
		\begin{subtable}[h]{0.5\textwidth}
			\centering
			\begin{tabular}{rrrr}
				\hline
				$n$ & $\mathcal{T}_\infty$ & $\mathcal{T}_1$ & KSB3 \\ 
				\hline
				50 & 0.19 & 0.31 & 0.12 \\ 
				100 & 0.22 & 0.29 & 0.11 \\ 
				200 & 0.27 & 0.36 & 0.11 \\ 
				500 & 0.50 & 0.54 & 0.12 \\ 
				1000 & 0.72 & 0.72 & 0.18 \\ 
				\hline
			\end{tabular}
			\caption{Test for $\mathcal{H}_0^R: G \geq_2 F$ (false)}
			\label{tw1r}
		\end{subtable}
		\caption{Rejection rates under $F\sim W(1.1,q_{1.1})$ and $G\sim W(1,1)$; independent samples}
		\label{tw1}			
	\end{table}
	
	\begin{table}[hbt!]
		\begin{subtable}[h]{0.5\textwidth}
			\centering
			\begin{tabular}{rrrr}
				\hline
				$n$ & $\mathcal{T}_\infty$ & $\mathcal{T}_1$ & KSB3 \\ 
				\hline
				50 & 0.04 & 0.05 & 0.12 \\ 
				100 & 0.04 & 0.02 & 0.13 \\ 
				200 & 0.02 & 0.00 & 0.12 \\ 
				500 & 0.01 & 0.00 & 0.11 \\ 
				1000 & 0.00 & 0.00 & 0.10 \\  
				\hline
			\end{tabular}
			\caption{Test for $\mathcal{H}_0: F \geq_2 G$ (true)}
			\label{tw2l}
		\end{subtable}
		\hfill
		\begin{subtable}[h]{0.5\textwidth}
			\centering
			\begin{tabular}{rrrr}
				\hline
				$n$ & $\mathcal{T}_\infty$ & $\mathcal{T}_1$ & KSB3 \\ 
				\hline
				50 & 0.29 & 0.45 & 0.14 \\ 
				100 & 0.37 & 0.50 & 0.13 \\ 
				200 & 0.56 & 0.66 & 0.15 \\ 
				500 & 0.87 & 0.88 & 0.26 \\ 
				1000 & 0.99 & 0.99 & 0.48 \\ 
				\hline
			\end{tabular}
			\caption{Test for $\mathcal{H}_0^R: G \geq_2 F$ (false)}
			\label{tw2r}
		\end{subtable}
		\caption{Rejection rates under $F\sim W(1.2,q_{1.2})$ and $G\sim W(1,1)$; independent samples}
		\label{tw2}
	\end{table}		
	
	\begin{table}[hbt!]
		\begin{subtable}[h]{0.5\textwidth}
			\centering
			\begin{tabular}{rrrr}
				\hline
				$n$ & $\mathcal{T}_\infty$ & $\mathcal{T}_1$ & KSB3 \\ 
				\hline
				50 & 0.03 & 0.03 & 0.12 \\ 
				100 & 0.02 & 0.00 & 0.12 \\ 
				200 & 0.01 & 0.00 & 0.11 \\ 
				500 & 0.01 & 0.00 & 0.11 \\ 
				1000 & 0.00 & 0.00 & 0.08 \\ 
				\hline
			\end{tabular}
			\caption{Test for $\mathcal{H}_0: F \geq_2 G$ (true)}
			\label{tw3l}
		\end{subtable}
		\hfill
		\begin{subtable}[h]{0.5\textwidth}
			\centering
			\begin{tabular}{rrrr}
				\hline
				$n$ & $\mathcal{T}_\infty$ & $\mathcal{T}_1$ & KSB3 \\ 
				\hline
				50 & 0.41 & 0.57 & 0.16 \\ 
				100 & 0.56 & 0.67 & 0.18 \\ 
				200 & 0.81 & 0.85 & 0.28 \\ 
				500 & 0.99 & 0.99 & 0.51 \\ 
				1000 & 1.00 & 1.00 & 0.88 \\ 
				\hline
			\end{tabular}
			\caption{Test for $\mathcal{H}_0^R: G \geq_2 F$ (false)}
			\label{tw3r}
		\end{subtable}
		\caption{Rejection rates under $F\sim W(1.3,q_{1.3})$ and $G\sim W(1,1)$; independent samples}
		\label{tw3}
	\end{table}
	
	\begin{table}[t!]
		\begin{subtable}[h]{0.5\textwidth}
			\centering
			\begin{tabular}{rrrr}
				\hline
				$n$ & $\mathcal{T}_\infty$ & $\mathcal{T}_1$ & KSB3 \\ 
				\hline
				50 & 0.25 & 0.17 & 0.43 \\ 
				100 & 0.43 & 0.15 & 0.59 \\ 
				200 & 0.58 & 0.10 & 0.74 \\ 
				500 & 0.89 & 0.08 & 0.98 \\ 
				1000 & 0.99 & 0.08 & 1.00 \\ 
				\hline
			\end{tabular}
			\caption{Test for $\mathcal{H}_0: F \geq_2 G$ (false)}
		\end{subtable}
		\hfill
		\begin{subtable}[h]{0.5\textwidth}
			\centering
			\begin{tabular}{rrrr}
				\hline
				$n$ & $\mathcal{T}_\infty$ & $\mathcal{T}_1$ & KSB3 \\ 
				\hline
				50 & 0.16 & 0.34 & 0.03 \\ 
				100 & 0.17 & 0.29 & 0.01 \\ 
				200 & 0.26 & 0.30 & 0.01 \\ 
				500 & 0.51 & 0.41 & 0.01 \\ 
				1000 & 0.84 & 0.68 & 0.02 \\  
				\hline
			\end{tabular}
			\caption{Test for $\mathcal{H}_0^R: G \geq_2 F$ (false)}
		\end{subtable}
		\caption{Rejection rates under $F=0.9F_{LN}(.;0.85;0.4)+0.1F_{LN}(0.4,0.9)$, $G=F_{LN}(.;0.86,0.6)$; independent samples. Note that, in case (a), the empirical power of $\mathcal{T}_1$ reaches 0.68 for $n=5000$.}
		\label{tln}	
	\end{table}
	
	\begin{table}[t!]
		\begin{subtable}[h]{0.5\textwidth}
			\centering
			\begin{tabular}{rrrr}
				\hline
				$n$ & $\mathcal{T}_\infty$ & $\mathcal{T}_1$ & KSB3 \\ 
				\hline
				50 & 0.02 & 0.00 & 0.38 \\ 
				100 & 0.04 & 0.00 & 0.52 \\ 
				200 & 0.05 & 0.00 & 0.63 \\ 
				500 & 0.11 & 0.00 & 0.88 \\ 
				1000 & 0.23 & 0.00 & 0.97 \\ 
				\hline
			\end{tabular}
			\caption{Test for $\mathcal{H}_0: F \geq_2 G$ (false)}
			\label{tsm1l}
		\end{subtable}
		\hfill
		\begin{subtable}[h]{0.5\textwidth}
			\centering
			\begin{tabular}{rrrr}
				\hline
				$n$ & $\mathcal{T}_\infty$ & $\mathcal{T}_1$ & KSB3 \\ 
				\hline
				50 & 0.56 & 0.77 & 0.06 \\ 
				100 & 0.76 & 0.87 & 0.03 \\ 
				200 & 0.97 & 0.98 & 0.02 \\ 
				500 & 1.00 & 1.00 & 0.01 \\ 
				1000 & 1.00 & 1.00 & 0.03 \\ 
				\hline
			\end{tabular}
			\caption{Test for $\mathcal{H}_0^R: G \geq_2 F$ (false)}
			\label{tsm1r}
		\end{subtable}
		\caption{Rejection rates under $F\sim\text{SM}(1.5,1.8)$, $G\sim\text{SM}(1,1.8)$; independent samples}
		\label{tsm1}
	\end{table}
	
	\begin{table}[hbt!]
		\begin{subtable}[h]{0.5\textwidth}
			\centering
			\begin{tabular}{rrrr}
				\hline
				$n$ & $\mathcal{T}_\infty$ & $\mathcal{T}_1$ & KSB3 \\ 
				\hline
				50 & 0.07 & 0.01 & 0.51 \\ 
				100 & 0.10 & 0.01 & 0.66 \\ 
				200 & 0.16 & 0.00 & 0.83 \\ 
				500 & 0.44 & 0.00 & 0.96 \\ 
				1000 & 0.81 & 0.00 & 1.00 \\ 
				\hline
			\end{tabular}
			\caption{Test for $\mathcal{H}_0: F \geq_2 G$ (false)}
			\label{tsm2l}
		\end{subtable}
		\hfill
		\begin{subtable}[h]{0.5\textwidth}
			\centering
			\begin{tabular}{rrrr}
				\hline
				$n$ & $\mathcal{T}_\infty$ & $\mathcal{T}_1$ & KSB3 \\ 
				\hline
				50 & 0.43 & 0.68 & 0.04 \\ 
				100 & 0.63 & 0.77 & 0.01 \\ 
				200 & 0.92 & 0.94 & 0.00 \\ 
				500 & 1.00 & 1.00 & 0.00 \\ 
				1000 & 1.00 & 1.00 & 0.00 \\ 
				\hline
			\end{tabular} 
			\caption{Test for $\mathcal{H}_0^R: G \geq_2 F$ (false)}
			\label{tsm2r}
		\end{subtable}
		\caption{Rejection rates under $F\sim\text{SM}(1.5,1.5)$, $G\sim\text{SM}(1,1.5)$; independent samples}
		\label{tsm2}
	\end{table}

	\begin{table}[hbt!]
		\begin{subtable}[h]{0.5\textwidth}
			\centering
			\begin{tabular}{rrrr}
				\hline
				$n$ & $\mathcal{T}_\infty$ & $\mathcal{T}_1$ & KSB3 \\ 
				\hline
				50 & 0.16 & 0.05 & 0.64 \\ 
				100 & 0.24 & 0.03 & 0.79 \\ 
				200 & 0.45 & 0.01 & 0.92 \\ 
				500 & 0.85 & 0.01 & 0.99 \\ 
				1000 & 0.99 & 0.01 & 1.00\\
				\hline
			\end{tabular}
			\caption{Test for $\mathcal{H}_0: F \geq_2 G$ (false)}
			\label{tsm3l}
		\end{subtable}
		\hfill
		\begin{subtable}[h]{0.5\textwidth}
			\centering
			\begin{tabular}{rrrr}
				\hline
				$n$ & $\mathcal{T}_\infty$ & $\mathcal{T}_1$ & KSB3 \\ 
				\hline
				50 & 0.27 & 0.52 & 0.03 \\ 
				100 & 0.48 & 0.61 & 0.00 \\ 
				200 & 0.77 & 0.83 & 0.01 \\ 
				500 & 0.99 & 0.99 & 0.00 \\ 
				1000 & 1.00 & 1.00 & 0.00 \\ 
				\hline
			\end{tabular}
			\caption{Test for $\mathcal{H}_0^R: G \geq_2 F$ (false)}
			\label{tsm3r}
		\end{subtable}
		\caption{Rejection rates under $F\sim\text{SM}(1.5,1.2)$, $G\sim\text{SM}(1,1.2)$; independent samples}
		\label{tsm3}
	\end{table}

	\begin{table}[hbt!]
		\begin{subtable}[h]{0.5\textwidth}
			\centering
			\begin{tabular}{rrrr}
				\hline
				$n$ & $\mathcal{T}_\infty$ & $\mathcal{T}_1$ & KSB3 \\ 
				\hline
				50 & 0.05 & 0.01 & 0.45 \\ 
				100 & 0.04 & 0.00 & 0.54 \\ 
				200 & 0.08 & 0.00 & 0.68 \\ 
				500 & 0.17 & 0.00 & 0.90 \\ 
				1000 & 0.32 & 0.00 & 0.99 \\  
				\hline
			\end{tabular}
			\caption{Test for $\mathcal{H}_0: F \geq_2 G$ (false)}
			
		\end{subtable}
		\hfill
		\begin{subtable}[h]{0.5\textwidth}
			\centering
			\begin{tabular}{rrrr}
				\hline
				$n$ & $\mathcal{T}_\infty$ & $\mathcal{T}_1$ & KSB3 \\ 
				\hline
				50 & 0.64 & 0.84 & 0.07 \\ 
				100 & 0.87 & 0.94 & 0.03 \\ 
				200 & 1.00 & 1.00 & 0.02 \\ 
				500 & 1.00 & 1.00 & 0.01 \\ 
				1000 & 1.00 & 1.00 & 0.04 \\ 
				\hline
			\end{tabular}
			\caption{Test for $\mathcal{H}_0^R: G \geq_2 F$ (false)}
		\end{subtable}
		\caption{Rejection rates under $F\sim\text{SM}(1.5,1.8)$, $G\sim\text{SM}(1,1.8)$, dependent samples with $\rho = 0.25$}\label{tsm125}
	\end{table}

	\begin{table}[hbt!]
		\begin{subtable}[h]{0.5\textwidth}
			\centering
			\begin{tabular}{rrrr}
				\hline
				$n$ & $\mathcal{T}_\infty$ & $\mathcal{T}_1$ & KSB3 \\ 
				\hline
				50 & 0.04 & 0.01 & 0.49 \\ 
				100 & 0.05 & 0.00 & 0.59 \\ 
				200 & 0.09 & 0.00 & 0.74 \\ 
				500 & 0.20 & 0.00 & 0.94 \\ 
				1000 & 0.44 & 0.00 & 1.00 \\ 
				\hline
			\end{tabular}
			\caption{Test for $\mathcal{H}_0: F \geq_2 G$ (false)}
		\end{subtable}
		\hfill
		\begin{subtable}[h]{0.5\textwidth}
			\centering
			\begin{tabular}{rrrr}
				\hline
				$n$ & $\mathcal{T}_\infty$ & $\mathcal{T}_1$ & KSB3 \\ 
				\hline
				50 & 0.75 & 0.93 & 0.07 \\ 
				100 & 0.94 & 0.99 & 0.04 \\ 
				200 & 1.00 & 1.00 & 0.02 \\ 
				500 & 1.00 & 1.00 & 0.02 \\ 
				1000 & 1.00 & 1.00 & 0.06 \\ 
				\hline
			\end{tabular}
			\caption{Test for $\mathcal{H}_0^R: G \geq_2 F$ (false)}
			
		\end{subtable}
		\caption{Rejection rates under $F\sim\text{SM}(1.5,1.8)$, $G\sim\text{SM}(1,1.8)$, dependent samples with $\rho = 0.5$}\label{tsm15}
	\end{table}
	

	\begin{table}[h!]
		\begin{subtable}[h]{0.5\textwidth}
			\centering
			\begin{tabular}{rrrr}
				\hline
				$n$ & $\mathcal{T}_\infty$ & $\mathcal{T}_1$ & KSB3 \\ 
				\hline
				50 & 0.03 & 0.00 & 0.58 \\ 
				100 & 0.08 & 0.00 & 0.70 \\ 
				200 & 0.13 & 0.00 & 0.87 \\ 
				500 & 0.32 & 0.00 & 0.99 \\ 
				1000 & 0.72 & 0.00 & 1.00 \\ 
				\hline
			\end{tabular}
			\caption{Test for $\mathcal{H}_0: F \geq_2 G$ (false)}
			
		\end{subtable}
		\hfill
		\begin{subtable}[h]{0.5\textwidth}
			\centering
			\begin{tabular}{rrrr}
				\hline
				$n$ & $\mathcal{T}_\infty$ & $\mathcal{T}_1$ & KSB3 \\ 
				\hline
				50 & 0.92 & 0.99 & 0.07 \\ 
				100 & 1.00 & 1.00 & 0.04 \\ 
				200 & 1.00 & 1.00 & 0.04 \\ 
				500 & 1.00 & 1.00 & 0.05 \\ 
				1000 & 1.00 & 1.00 & 0.15 \\ 
				\hline
			\end{tabular}
			\caption{Test for $\mathcal{H}_0^R: G \geq_2 F$ (false)}
			
		\end{subtable}
		\caption{Rejection rates under $F\sim\text{SM}(1.5,1.8)$, $G\sim\text{SM}(1,1.8)$, dependent samples with $\rho = 0.75$}\label{tsm175}
	\end{table}
	
	\begin{table}[hbt!]
		\begin{subtable}[h]{0.5\textwidth}
			\centering
			\begin{tabular}{rrrr}
				\hline
				$n$ & $\mathcal{T}_\infty$ & $\mathcal{T}_1$ & KSB3 \\ 
				\hline
				50 & 0.19 & 0.05 & 0.70 \\ 
				100 & 0.28 & 0.02 & 0.80 \\ 
				200 & 0.51 & 0.01 & 0.92 \\ 
				500 & 0.88 & 0.01 & 0.99 \\ 
				1000 & 1.00 & 0.02 & 0.99 \\
				\hline
			\end{tabular}
			\caption{Test for $\mathcal{H}_0: F \geq_2 G$ (false)}
			
		\end{subtable}
		\hfill
		\begin{subtable}[h]{0.5\textwidth}
			\centering
			\begin{tabular}{rrrr}
				\hline
				$n$ & $\mathcal{T}_\infty$ & $\mathcal{T}_1$ & KSB3 \\ 
				\hline
				50 & 0.34 & 0.58 & 0.03 \\ 
				100 & 0.56 & 0.71 & 0.01 \\ 
				200 & 0.89 & 0.91 & 0.01 \\ 
				500 & 1.00 & 1.00 & 0.00 \\ 
				1000 & 1.00 & 1.00 & 0.00 \\ 
				\hline
			\end{tabular}
			\caption{Test for $\mathcal{H}_0^R: G \geq_2 F$ (false)}
			
		\end{subtable}
		\caption{Rejection rates under $F\sim\text{SM}(1.5,1.2)$, $G\sim\text{SM}(1,1.2)$, dependent samples with $\rho = 0.25$}\label{tsm225}
	\end{table}
	
	\begin{table}[hbt!]
		\begin{subtable}[h]{0.5\textwidth}
			\centering
			\begin{tabular}{rrrr}
				\hline
				$n$ & $\mathcal{T}_\infty$ & $\mathcal{T}_1$ & KSB3 \\ 
				\hline
				50 & 0.21 & 0.05 & 0.76 \\ 
				100 & 0.38 & 0.01 & 0.85 \\ 
				200 & 0.64 & 0.01 & 0.95 \\ 
				500 & 0.96 & 0.01 & 0.99 \\ 
				1000 & 1.00 & 0.03 & 0.99 \\ 
				\hline
			\end{tabular}
			\caption{Test for $\mathcal{H}_0: F \geq_2 G$ (false)}
			
		\end{subtable}
		\hfill
		\begin{subtable}[h]{0.5\textwidth}
			\centering
			\begin{tabular}{rrrr}
				\hline
				$n$ & $\mathcal{T}_\infty$ & $\mathcal{T}_1$ & KSB3 \\ 
				\hline
				50 & 0.42 & 0.68 & 0.03 \\ 
				100 & 0.69 & 0.84 & 0.00 \\ 
				200 & 0.96 & 0.97 & 0.00 \\ 
				500 & 1.00 & 1.00 & 0.00 \\ 
				1000 & 1.00 & 1.00 & 0.00 \\ 
				\hline
			\end{tabular}
			\caption{Test for $\mathcal{H}_0^R: G \geq_2 F$ (false)}
		\end{subtable}
		\caption{Rejection rates under $F\sim\text{SM}(1.5,1.2)$, $G\sim\text{SM}(1,1.2)$, dependent samples with $\rho = 0.5$}\label{tsm25}
	\end{table}
	
	\begin{table}[hbt!]
		\begin{subtable}[h]{0.5\textwidth}
			\centering
			\begin{tabular}{rrrr}
				\hline
				$n$ & $\mathcal{T}_\infty$ & $\mathcal{T}_1$ & KSB3 \\ 
				\hline
				50 & 0.28 & 0.04 & 0.86 \\ 
				100 & 0.57 & 0.01 & 0.93 \\ 
				200 & 0.85 & 0.02 & 0.98 \\ 
				500 & 1.00 & 0.02 & 0.99 \\ 
				1000 & 1.00 & 0.09 & 1.00 \\
				\hline
			\end{tabular}
			\caption{Test for $\mathcal{H}_0: F \geq_2 G$ (false)}
			
		\end{subtable}
		\hfill
		\begin{subtable}[h]{0.5\textwidth}
			\centering
			\begin{tabular}{rrrr}
				\hline
				$n$ & $\mathcal{T}_\infty$ & $\mathcal{T}_1$ & KSB3 \\ 
				\hline
				50 & 0.61 & 0.88 & 0.02 \\ 
				100 & 0.91 & 0.97 & 0.00 \\ 
				200 & 1.00 & 1.00 & 0.00 \\ 
				500 & 1.00 & 1.00 & 0.00 \\ 
				1000 & 1.00 & 1.00 & 0.00 \\ 
				\hline
			\end{tabular}
			\caption{Test for $\mathcal{H}_0^R: G \geq_2 F$ (false)}
		\end{subtable}
		\caption{Rejection rates under $F\sim\text{SM}(1.5,1.2)$, $G\sim\text{SM}(1,1.2)$, dependent samples with $\rho = 0.75$}\label{tsm275}
	\end{table}
	
	\begin{table}[hbt!]
		\begin{subtable}[h]{0.5\textwidth}
			\centering
			\begin{tabular}{rrrr}
				\hline
				$n$ & $\mathcal{T}_\infty$ & $\mathcal{T}_1$ & KSB3 \\ 
				\hline
				50 & 0.11 & 0.14 & 0.24 \\ 
				100 & 0.22 & 0.16 & 0.41 \\ 
				200 & 0.44 & 0.21 & 0.64 \\ 
				500 & 0.86 & 0.42 & 0.92 \\ 
				1000 & 1.00 & 0.82 & 0.95 \\ 
				\hline
			\end{tabular}
			\caption{Test for $ {H}_0^1: F \geq_1 G $ (false)}
			
		\end{subtable}
		\hfill
		\begin{subtable}[h]{0.5\textwidth}
			\centering
			\begin{tabular}{rrrr}
				\hline
				$n$ & $\mathcal{T}_\infty$ & $\mathcal{T}_1$ & KSB3 \\ 
				\hline
				50 & 0.03 & 0.29 & 0.19 \\ 
				100 & 0.10 & 0.38 & 0.19 \\ 
				200 & 0.31 & 0.53 & 0.14 \\ 
				500 & 0.92 & 0.92 & 0.05 \\ 
				1000 & 1.00 & 1.00 & 0.00 \\    
				\hline
			\end{tabular}
			\caption{Test for $({H}_0^1)^R: G \geq_1 F$ (false)}
			
		\end{subtable}
		\caption{FSD test. Rejection rates under $F\sim\text{SM}(1.5,1.2)$, $G\sim\text{SM}(1,1.2)$}\label{tfsd1}
	\end{table}
	

	\begin{table}[hbt!]
		\begin{subtable}[h]{0.5\textwidth}
			\centering
			\begin{tabular}{rrrr}
				\hline
				$n$ & $\mathcal{T}_\infty$ & $\mathcal{T}_1$ & KSB3 \\ 
				\hline
				50 & 0.08 & 0.09 & 0.15 \\ 
				100 & 0.11 & 0.06 & 0.24 \\ 
				200 & 0.20 & 0.05 & 0.39 \\ 
				500 & 0.64 & 0.10 & 0.85 \\ 
				1000 & 0.96 & 0.29 & 0.97 \\
				\hline
			\end{tabular}
			\caption{Test for ${H}_0^1: F \geq_1 G$ (false)}
			
		\end{subtable}
		\hfill
		\begin{subtable}[h]{0.5\textwidth}
			\centering
			\begin{tabular}{rrrr}
				\hline
				$n$ & $\mathcal{T}_\infty$ & $\mathcal{T}_1$ & KSB3 \\ 
				\hline
				50 & 0.08 & 0.44 & 0.32 \\ 
				100 & 0.19 & 0.56 & 0.44 \\ 
				200 & 0.52 & 0.80 & 0.52 \\ 
				500 & 0.99 & 0.99 & 0.39 \\ 
				1000 & 1.00 & 1.00 & 0.22 \\  
				\hline
			\end{tabular}
			\caption{Test for $({H}_0^1)^R: G \geq_1 F$ (false)}
			
		\end{subtable}
		\caption{FSD test. Rejection rates under $F\sim\text{SM}(1.5,1.5)$, $G\sim\text{SM}(1,1.5)$}\label{tfsd2}
	\end{table}

	\begin{table}[hbt!]
		\begin{subtable}[h]{0.5\textwidth}
			\centering
			\begin{tabular}{rrrr}
				\hline
				$n$ & $\mathcal{T}_\infty$ & $\mathcal{T}_1$ & KSB3 \\ 
				\hline
				50 & 0.03 & 0.02 & 0.07 \\ 
				100 & 0.04 & 0.02 & 0.13 \\ 
				200 & 0.09 & 0.01 & 0.24 \\ 
				500 & 0.36 & 0.01 & 0.60 \\ 
				1000 & 0.77 & 0.02 & 0.92 \\ 
				\hline
			\end{tabular}
			\caption{Test for ${H}_0^1: F \geq_1 G$ (false)}
			
		\end{subtable}
		\hfill
		\begin{subtable}[h]{0.5\textwidth}
			\centering
			\begin{tabular}{rrrr}
				\hline
				$n$ & $\mathcal{T}_\infty$ & $\mathcal{T}_1$ & KSB3 \\ 
				\hline
				50 & 0.10 & 0.59 & 0.47 \\ 
				100 & 0.30 & 0.70 & 0.63 \\ 
				200 & 0.74 & 0.91 & 0.78 \\ 
				500 & 0.99 & 1.00 & 0.78 \\ 
				1000 & 1.00 & 1.00 & 0.67 \\   
				\hline
			\end{tabular}
			\caption{Test for $({H}_0^1)^R: G \geq_1 F$ (false)}
			
		\end{subtable}
		\caption{FSD test. Rejection rates under $F\sim\text{SM}(1.5,1.8)$, $G\sim\text{SM}(1,1.8)$}\label{tfsd3}
	\end{table}

			\section*{Appendix B: Proofs}
			\label{app_proofs}
			
			\begin{proof}[Calculations of Example~\ref{ex weibull}.]
				{The unscaled Lorenz curve of $ F $ is}
				$$L_F(p)=b \left(\Gamma \left(1+\frac{1}{a}\right)-\Gamma
				\left(1+\frac{1}{a},-\log (1-p)\right)\right),$$
				{while} 
				$$L_G(p)=p+(1-p) \log (1-p).$$ 
				It is well known \citep[e.g.][]{goldie} that $L_G$ can be expressed as $L_G(p)=M_G\circ G^{-1}(p)$, with $$M_G(x)=\int_0^x t \  dG(t)=1-e^{-x} (x+1), \qquad x\geq0.$$ 
				Noting that $M_G(x)\leq\mu_G=1$ for any $x\geq0$, this function can be inverted using the Lambert $W_{-1}$ function \citep{lambert}, that is 
				$M_G^{-1}(t)=-1-W_{-1}\left((t-1)/e\right)$. 
				Accordingly, $$L_G^{-1}(t)=G\circ M_G^{-1}(t)=1-\exp\left({1+W_{-1}\left(\frac{t-1}{e}\right)}\right).$$Finally, by composition, we obtain the expression of $Z$ in Example~\ref{ex weibull}.
			\end{proof}

			\begin{proof}[Proof of Proposition~\ref{properties}.]\begin{enumerate}\item This follows from the properties of the $L^p$ norm.\item If $v_2(x)\leq0,\forall x\in[0,1]$ then $v_1(x)-v_2(x)\geq v_1(x),\forall x\in[0,1]$ which implies $(v_1(x)-v_2(x))_+^p\geq (v_1(x))_+^p,\forall x\in[0,1]$ and therefore $||(v_1-v_2)_+||_p\geq||(v_1)_+||_p$ by monotonicity of integrals.
					\item The proof is the same as in Lemma 2 of \cite{bdlorenz} and relies on the fact that $v_1\in C[0,1]$.
					\item Minkowski's inequality implies that, for some pair of functions $u,v\in C[0,1]$, $||u||_p=||(u-v)+v||_p\leq ||u-v||_p+||v||_p$, so that $||u||_p-||v||_p\leq ||u-v||_p$, and similarly, $||u-v||_p\geq||v||_p-||u||_p$, therefore, $|||u||_p-||v||_p|\leq ||u-v||_p$. Then $$\lvert||(v_1)_+||_p-||(v_2)_+||_p\lvert\leq ||(v_1)_+-(v_2)_+||_p\leq ||v_1-v_2||_p\leq ||v_1-v_2||_\infty,$$
					where the second inequality follows from the fact that, for every $x\in[0,1]$,  $|(v_1(x))_+-(v_2(x))_+|\leq |v_1(x)-v_2(x)|$.
					\item The proof follows from absolute homogeneity of the $L^p$ norm.
					\item Let $\beta\in[0,1]$. By convexity of the function $(\cdot)_+$, Minkowski's inequality, and absolute homogeneity of the $L^p$ norm, \begin{align*}\mathcal{T}_p(\beta (v_2)+(1-\beta)v_1)=||(\beta v_2+(1-\beta)v_1)_+||_p\leq ||\beta (v_2)_++(1-\beta)(v_1)_+||_p\\
						\leq\beta|| (v_2)_+||_p+(1-\beta)||(v_1)_+||_p=\beta \mathcal{T}_p(v_2)+(1-\beta)\mathcal{T}_p(v_1).\end{align*}
					\item This follows from basic properties of $L^p$ norms.
				\end{enumerate}
				
			\end{proof}
			
			\begin{proof}[Proof of Proposition~\ref{P1}.]
				$\widetilde{\LPP }_{n,m}-\LPP _{n,m}$ can be expressed as
				$$\widetilde{L}^{-1}_{G_m}\circ \widetilde{L}_{F_n}-{L}^{-1}_{G_m}\circ {L}_{F_n}=(\widetilde{L}^{-1}_{G_m}\circ \widetilde{L}_{F_n}-{\widetilde{L}}^{-1}_{G_m}\circ {L}_{F_n})+(\widetilde{L}^{-1}_{G_m}\circ{L}_{F_n}-{L}^{-1}_{G_m}\circ {L}_{F_n}).$$
				For the first summand, which is the difference between two step functions, we have $\widetilde{L}^{-1}_{G_m}\circ \widetilde{L}_{F_n}(p)\geq{\widetilde{L}}^{-1}_{G_m}\circ {L}_{F_n} (p)$ for every $p\in[0,1]$, since $ \widetilde{L}_{F_n}(p)\geq {L}_{F_n} (p)$ for every $p\in[0,1]$. Moreover, $\widetilde{L}^{-1}_{G_m}\circ \widetilde{L}_{F_n}(k/n)={\widetilde{L}}^{-1}_{G_m}\circ {L}_{F_n} (k/n)$ for $k=0,...,n$, while, within each interval $((k-1)/n,k/n)$, the difference $\widetilde{L}^{-1}_{G_m}\circ \widetilde{L}_{F_n}(p)-{\widetilde{L}}^{-1}_{G_m}\circ {L}_{F_n} (p)$ is bounded above the height of the jumps of $\widetilde{L}^{-1}_{G_m}$, that is, $1/m$. 
				For the latter summand, $\widetilde{L}^{-1}_{G_m}\circ {L}_{F_n}-{L}^{-1}_{G_m}\circ {L}_{F_n}\in[-\frac1m,0]$, since clearly ${L}^{-1}_{G_m}\circ{L}_{F_n}$ is the linear interpolator of the jump points of the step function $\widetilde{L}^{-1}_{G_m}\circ {L}_{F_n}$. Hence, the result follows.
			\end{proof}
			
			\begin{proof}[Proof of Proposition~\ref{P2}.]
				As proved in Theorem 10.1 and Theorem 13.2 of \cite{csorgo2013}, $\widetilde{L}_{G_m}^{-1}$ and $\widetilde{L}_{F_n}$ converge strongly and uniformly to $L_{G}^{-1}$ and $L_{F}$, respectively. Since $L_G^{-1}$ is uniformly continuous in $[0,\infty)$ and $\sup_p|L_{F_n}(p)-L_{F}(p)|\rightarrow 0$ almost surely, we obtain that $\sup_p|L_G^{-1}\circ \widetilde{L}_{F_n}(p)-L_G^{-1}\circ L_{F}(p)|\rightarrow 0$ almost surely. Then, for every $p\in(0,1)$,
				\begin{align*}
					|\widetilde{L}_{G_{m}}^{-1}\circ \widetilde{L}_{F_n}(p)-L_G^{-1}\circ L_{F}(p)|
					&\leq|\widetilde{L}_{G_{m}}^{-1}\circ \widetilde{L}_{F_n}(p)-L_G^{-1}\circ \widetilde{L}_{F_n}(p)|+|L_{G}^{-1}\circ \widetilde{L}_{F_n}(p)-L_G^{-1}\circ L_{F}(p)|\\
					&\leq\sup_p|\widetilde{L}^{-1}_{G_m}(p)-L^{-1}_{G}(p)|+\sup_p|L_G^{-1}\circ \widetilde{L}_{F_n}(p)-L_G^{-1}\circ L_{F}(p)|.
				\end{align*}
				Since both terms in the right-hand side                                                                                                                                                                                                                                                                                                                                                                                                                                      converge to 0 with probability 1, we obtain that $\widetilde{\LPP }_{n,m}$ converges strongly and uniformly to $\LPP $ in $[0,1]$. By Proposition~\ref{P1}, {$|\widetilde{\LPP }_{n,m}-\LPP _{n,m}|\rightarrow0$ for $n \rightarrow \infty$ and $ m \rightarrow \infty $}, therefore the same property is satisfied by ${\LPP }_{n,m}$.
			\end{proof}
			
			\begin{proof}[Proof of Theorem~\ref{LPP}.]
				Let $\mathbb{L}$ be the space of maps $z:[0,\infty)\rightarrow\mathbb{R}$ with $\lim_{x\rightarrow -\infty}z(x)=0$ and $\lim_{x\rightarrow\infty}z(x)=1$, and the norm $||z||_\mathbb{L}=\max\{||z||_\infty,||1-z||_1\}$. As shown by \cite{kaji}, under assumption i), the map $\phi(F)=F^{-1}$, from CDFs to quantile functions, is Hadamard differentiable at $F$, tangentially to the set $\mathbb{L}_0$ of continuous functions in $\mathbb{L}$, with derivative map
				$$\phi_F'(z)=-(z\circ F^{-1})(F^{-1})'. $$
				The linear map ${\psi }(F^{-1})=\int_0^. F^{-1}(t)dt$ coincides with its Hadamard derivative. Accordingly, by the chain rule \citep[Lemma 3.9.3]{vw}, the composition map $\psi \circ \phi:F\rightarrow L_F$ is also Hadamard differentiable at $F$ tangentially to $\mathbb{L}_0$, with derivative $$(\psi \circ\phi)'_F(z)=\psi '_{\phi(F)}\circ \phi_F'(z)=-\int_0^. z\circ F^{-1}(p)dF^{-1}(p).$$
				Now, observe that $$\begin{pmatrix}\sqrt{n}(F_n-F) \\ \sqrt{m}(G_m-G)\end{pmatrix}\rightsquigarrow \begin{pmatrix}\mathcal{B}_1\circ F \\ \mathcal{B}_2\circ G\end{pmatrix}\text{  in }\mathbb{L}\times\mathbb{L},$$
				as shown in Lemma 5.1 of \cite{sunbeare}. Then, the functional delta method \citep[Theorem 3.9.13]{vw} implies the joint weak convergence 
				\begin{multline}
					\label{lor process}
					\begin{pmatrix}\sqrt{n}(L_{F_n}-L_F) \\ \sqrt{m}(L_{G_m}-L_G)\end{pmatrix}=\begin{pmatrix}\sqrt{n}(\psi \circ\phi({F_n})-\psi \circ\phi({F})) \\ \sqrt{m}(\psi \circ\phi({G_m})-\psi \circ\phi({G}))\end{pmatrix} \rightsquigarrow\begin{pmatrix} (\psi \circ\phi)'_F(\mathcal{B}_1\circ F) \\ (\psi \circ\phi)'_G(\mathcal{B}_2\circ G)\end{pmatrix}\\=
					\begin{pmatrix}- \int_0^.\mathcal{B}_1(t)dF^{-1}(t) \\- \int_0^.\mathcal{B}_2(t)dG^{-1}(t)\end{pmatrix} =:\begin{pmatrix}\mathcal{L}_F\\ \mathcal{L}_G\end{pmatrix}\text{  in }C[0,1]\times C[0,1].
				\end{multline}
				Now, consider the process $\sqrt{m}(L^{-1}_{G_m}(t)-L^{-1}_G(t))$, for $t\in[0,\mu_G]$. The LC $L_G$ is increasing and continuous on $[0,1]$, therefore the inverse function $L_G^{-1}$ is increasing and continuous on $[0,\mu_G]$, moreover the derivative $L_G'=G^{-1}$ is strictly positive in the unit interval (note that assumption ii) entails that $G^{-1}(0)=c>0$). Then, by the inverse map theorem \cite[Lemma 3.9.20]{vw} the map $\eta:L_G\rightarrow L^{-1}_G$ is Hadamard-differentiable at $L_G$, tangentially to the set of bounded functions on $[0,1]$, with derivative $$\eta'_{L_G}(z)=-\frac{z\circ L_G^{-1}}{L_G'\circ L_G^{-1}}=-\frac{z\circ L_G^{-1}}{G^{-1}\circ L_G^{-1}}.$$
				Since $r_n/n\rightarrow 1-\lambda$ and $r_n/m\rightarrow \lambda$, by (\ref{lor process}) and the functional delta method, the above result implies
				\begin{equation}\label{conc process}\sqrt{r_n}\begin{pmatrix} L_{F_n}-L_F \\ L_{G_m}^{-1}-L_G^{-1} \end{pmatrix}\rightsquigarrow \begin{pmatrix}\lambda  \mathcal{L}_F \\ (1-\lambda)\eta'_{L_G}(\mathcal{L}_G) \end{pmatrix}=:\begin{pmatrix} \lambda \mathcal{L}_F \\ (1-\lambda) \mathcal{C}_G \end{pmatrix} \text{  in }C[0,1]\times C[0,\mu_G],\end{equation}
				where $\mathcal{C}_G$ is defined as
				$$\mathcal{C}_G(t)=\frac{\int_0^{L_G^{-1}(t)}\mathcal{B}_2(p)dG^{-1}(p)}{G^{-1}\circ L_G^{-1}(t)}, \qquad t\in[0,\mu_G].$$

				Now, consider the maps $\pi=\psi \circ \phi:F\rightarrow L_F$, $\theta:h\rightarrow L_G^{-1}\circ h$ and the composition map $\zeta:C[0,1]\times C[0,\mu_G]\rightarrow C[0,\nu]$ defined by $\zeta(\pi,\theta)(x)=\theta\circ \pi(x)$. 
				Recall that the Hadamard derivative of $\theta$ is $\theta'_{L_F}(\alpha)=((L_G^{-1})'\circ L_F)\alpha$, and $L_G^{-1}$ is uniformly norm-bounded, since $(L_G^{-1})'\leq 1/c$, therefore we can apply Lemma 3.2.27 of \cite{vw}, which establishes that $\zeta$ is Hadamard differentiable at $(\pi,\theta)$, tangentially to the set $C[0,1]\times UC[0,\mu_G]$, where $UC[0,\mu_G]$ is the family of uniformly continuous functions on $[0,\mu_G]$, with derivative
				$$\zeta'_{\pi,\theta}(\alpha,\beta)(x)=\beta\circ \pi(x)+\theta'_{\pi(x)}(\alpha(x))=\beta\circ L_F(x)+\theta'_{L_F}(\alpha(x)) =\beta\circ L_F(x)+  ((L_G^{-1})'\circ L_F)\alpha(x).$$
				Now, since $\LPP =\zeta(L_F,L^{-1}_G)$, using (\ref{conc process}), the functional delta method and the Hadamard differentiability of the composition map $\zeta$ give
				\begin{align*}
					\sqrt{r_n}(\LPP _{n,m}-\LPP )\1{[0,\nu]} &=\sqrt{r_n}(\zeta(L_{F_n},L^{-1}_{G_m})-\zeta(L_F,L^{-1}_G) )\1{[0,\nu]}\\
					&\rightsquigarrow \zeta'_{ L_F,L^{-1}_G}(\lambda\mathcal{L}_F,(1-\lambda)\mathcal{C}_G)
					=\sqrt{\lambda}\mathcal{C}_G\circ L_F+  \sqrt{1-\lambda} \frac{\mathcal{L}_F}{G^{-1}\circ L_G^{-1}\circ L_F}\\
					&=\frac{-\sqrt{1-\lambda}\int_0^.\mathcal{B}_1dF^{-1}(p)+\sqrt{\lambda}\int_0^\LPP \mathcal{B}_2(p)dG^{-1}(p)  }{G^{-1}\circ\LPP } \quad \text{in } C[0,\nu],
				\end{align*}
				which implies the statement since $\LP _n\1{(\nu,1]}\rightsquigarrow 0$.
			\end{proof}
			
			\begin{proof}[Proof of Lemma~\ref{lemmatest}.]
				Bear in mind that $\LP \1{[0,\nu]}$ is a mean-zero Gaussian process since it is obtained by integrating and normalizing Gaussian processes. Under $\mathcal{H}_0$, $p-\LPP (p)\leq 0,\forall p\in[0,1]$, hence $$\sqrt{r_n} \ \mathcal{T}(I-\LPP _{n,m})\leq \sqrt{r_n} \ \mathcal{T}(\LPP -\LPP _{n,m})=\mathcal{T}(\sqrt{r_n}(\LPP -\LPP _{n,m}))\rightsquigarrow \mathcal{T}(\LP ),$$ where the last step follows from the continuous mapping theorem, since the map $\mathcal{T}( f)$ satisfies $|\mathcal{T}( f)-\mathcal{T}( g)|\leq ||f-g||_\infty$, where $f,g$ are continuous functions on the unit interval, as proved in Lemma 2 of \cite{bdlorenz}. 
				The $(1-\alpha)$ quantile of the distribution of $\mathcal{T}(\LP )$ is positive, finite, and unique because $\LP $ is a mean zero Gaussian process, so the proof follows by the same arguments used in the proof of Lemma 4 in \cite{bdlorenz}. Since, by Proposition~\ref{P2}, $\LPP _{n,m}$ converges strongly and uniformly to $\LPP $, under $\mathcal{H}_1$ we have $\mathcal{T} (I-\LPP _{n,m})\rightarrow_p \mathcal{T} (I-\LPP )>0$. Finally, multiplying by $\sqrt{r_n}$, we obtain the second result.
			\end{proof}
			
			\begin{proof}[Proof of Proposition~\ref{ptest}.]
				As proved in Lemma 5.2 of \cite{sunbeare},
				$$\begin{pmatrix} \sqrt{n}({F^*_n}-{F_n} )\\ \sqrt{m}({G^*_n}-{G_n}) \end{pmatrix}\vc{\rightsquigarrow}{as*}{M} \begin{pmatrix} \mathcal{B}_1\circ F \\ \mathcal{B}_2\circ G \end{pmatrix} \text{  in }\mathbb{L}\times \mathbb{L},$$
				where $\vc{\rightsquigarrow}{as*}{M}$ denotes weak convergence conditional on the data a.s., see \cite[p.20]{kosorok}.
				The proof of Theorem~\ref{LPP} establishes the Hadamard-differentiability of the maps $\psi \circ \phi:F\rightarrow L_F$ and $\eta\circ \psi \circ \phi:G\rightarrow L_G^{-1}$, so that the functional delta method for the bootstrap implies
				$$\sqrt{r_n}\begin{pmatrix} L_{F^*_n}-L_{F_n} \\ L_{G^*_n}^{-1}-L_{G_n}^{-1} \end{pmatrix}\vc{\rightsquigarrow}{P}{M} \begin{pmatrix} \lambda \mathcal{L}_F \\ (1-\lambda) \mathcal{C}_G \end{pmatrix} \text{  in }C[0,1]\times C[0,\mu_G],$$
				where $\vc{\rightsquigarrow}{P}{M}$ denotes weak convergence conditional on the data in probability \citep[see][p.20]{kosorok}.
				Using the Hadamard differentiability of the composition map $\zeta(L_F,L_G^{-1})=L_G^{-1}\circ L_F$, the functional delta method for bootstrap implies $\sqrt{r_n}(\LPP _{n,m}^*-\LPP _{n,m})\vc{\rightsquigarrow}{P}{M} \LP $, which entails that $\mathcal{T}(\sqrt{r_n}(\LPP_{n,m}-\LPP^* _{n,m}))\vc{\rightsquigarrow}{P}{M} \mathcal{T}( -\LP )=_d \mathcal{T}( \LP )$ by the continuous mapping theorem. The test rejects the null hypothesis if the test statistic exceeds the bootstrap threshold $c^*_n(\alpha)=\inf\{y: P(\sqrt{r_n} \ \mathcal{T}(\LPP _{n,m}-\LPP ^*_{k;n,m})>y|{\mathcal{X},\mathcal{Y}})\leq \alpha\}$, but the weak convergence result implies $c_n^*(\alpha)\rightarrow_p c(\alpha)=\inf\{y: P(\mathcal{T}({\LP })>y)\leq \alpha\}$. Hence, Lemma~\ref{lemmatest} yields the result.
			\end{proof}
			
			\begin{proof}[Proof of Theorem~\ref{TT}.]
				Integrating by substitution, we can see that $X\geq^T_u Y$ if and only if $u(X)\geq_2 u(Y)$, since $P(u(X)\leq t)=F_X \circ u^{-1}(t)$, and similarly for $Y$. Hence, by setting $\phi=g\circ u$, the proof follows from the classic characterisation of SSD, since $\E(g\circ u(X))\geq\E(g\circ u(Y))$, for any increasing concave function $g$.
			\end{proof}
			
			\begin{proof}[Proof of Theorem~\ref{T7}.]
				Point 1) follows from the fact that $u_1(X)\geq_2 u_1(Y)$ implies $\E(u_2\circ u_1^{-1}\circ u_1(X))=\E(u_2(X))\geq \E(u_2(Y))$, because the composition $u_2\circ u_1^{-1}$ is concave by construction. The ``only if" part of point 2) is trivial. The ``if" part follows from the characterisation of FSD, taking into account that the equivalent condition of Theorem~\ref{TT}, that is, $\E(\phi(X))\geq\E(\phi(Y)),\forall \phi\leq_c u$, for every $u\in\mathcal{U}$, implies that such an inequality holds just for every increasing $\phi\in \mathcal{U}$. Since any increasing function may be approximated by a sequence of functions in $\mathcal{U}$, we have $X\geq_1 Y$.
			\end{proof}
			
			\begin{proof}[Proof of Theorem~\ref{T8}.] 
				1. 
				$X\geq_{1+1/\theta}^TY$ con be expressed as
				\begin{equation}\label{tsd}
					\int_{-\infty}^x F(t)du(t)\leq \int_{-\infty}^x G(t)du(t), \qquad \forall x.
				\end{equation}
				Integrating by parts and by substitution, we obtain that, for both $ H=F $ and $ H=G $,
				$$\mathcal{I}^\theta_H(x)=\int_0^x H(t)du_\theta(t)=u_\theta(x)H(x)-\int_0^x u_\theta(t)dH(t)=u_\theta(x)H(x)-\int_{0}^{H(x)} u_\theta\circ H^{-1}(y)dy.$$
				Hence,
				$$\frac{\mathcal{I}^\theta_H(x)}{u_\theta(x)}=H(x)-\int_{0}^{H(x)} \frac{u_\theta\circ H^{-1}(y)}{u_\theta(x)}dy=H(x)-\int_{0}^{H(x)}\left(\frac{H^{-1}(y)}{x}\right)^{\theta} dy\rightarrow H(x),$$
				by the Lebesgue dominated convergence theorem, recalling that ${H^{-1}(y)}/{x}\leq 1$ as $y\leq H(x)$. Now, it is readily seen that $X\geq^T_{1+1/\theta} Y$ if and only if ${\mathcal{I}^\theta_F(x)}/{u_\theta(x)} \leq  {\mathcal{I}^\theta_G(x)}/{u_\theta(x)}$ for any $x$, which implies the result.
				
				2.
				Let $x_1,...,x_n$ and $y_1,...,y_m$ be ordered realisations from $X$ and $Y$, respectively. By properties of the power function, there exists some number $\theta_0$ such that, for $\theta>\theta_0$, $(1/n)\sum_{k=1}^ix_k^{\theta}\in((1/m)\sum_{k=1}^{j-1}y_k^{\theta},(1/m)\sum_{k=1}^jy_k^{\theta})$ if and only if ${x_i}/n \in({y_{j-1}}/m,{y_j}/m)$. In fact,
				\begin{equation*} x_i\left(\frac1n(\sum_{k=1}^{i-1}\left(\tfrac{x_k}{x_i}\right)^{\theta}+1)\right)\in
					\left(y_{j-1}\left(\frac1m(\sum_{k=1}^{j-2}\left(\tfrac{y_k}{y_{j-1}}\right)^{\theta}+1)\right),y_j\left(\frac1m(\sum_{k=1}^{j-1}\left(\tfrac{y_k}{y_j}\right)^{\theta}+1)\right)\right).
				\end{equation*}
				Accordingly, for any $i=1,...,n$ and $\theta>\theta_0$, $\widetilde{\LPP }_{n,m}^\theta$ returns $ j/m$ if ${x_i}/n \in({y_{j-1}}/m,{y_j}/m)$, which coincides with the P-P plot $G_m\circ F_n^{-1}$.
		\end{proof}

			
			\section*{Funding}
			This research was supported by the Italian funds ex-MURST60\%. T.L. was also supported by the Czech Science Foundation (GACR) under the project 20-16764S and by V\v{S}B-TU Ostrava (SGS project SP2021/15).
			
			\vspace{.1in}
			\noindent \textit{Conflict of interest:} None declared.


\bibliographystyle{elsarticle-harv}
\bibliography{ssd_biblio}

\textbf{Corresponding Author.} Tommaso Lando, e-mail: tommaso.lando@unibg.it

\




\end{document}